\documentclass[12pt]{amsart}
\usepackage{amscd,amsmath,amsthm,amssymb,graphics}
\usepackage{slashbox}
\usepackage{pgf,tikz}
\usetikzlibrary{arrows}
\usepackage{lineno}

\def\NZQ{\Bbb}               
\def\NN{{\NZQ N}}

\def\ZZ{{\NZQ Z}}

\def\AA{{\NZQ A}}

\def\frk{\frak}               

\def\mm{{\frk m}}

\def\Phi{{\frk n}}
\def\Phi{{\frk N}}
\def\MA{{\mathcal A}}
\def\MB{{\mathcal B}}

\def\MP{{\mathcal P}}
\def\MQ{{\mathcal Q}}

\def\MK{{\mathcal K}}
\def\MF{{\mathcal F}}

\def\MY{{\mathcal Y}}

\def\ab{{\bold a}}
\def\bb{{\bold b}}

\def\cb{{\bold c}}

\def\eb{{\bold e}}
\def\opn#1#2{\def#1{\operatorname{#2}}} 
\opn\chara{char} \opn\length{\ell} \opn\pd{pd} \opn\rk{rk}
\opn\projdim{proj\,dim} \opn\injdim{inj\,dim} \opn\rank{rank}
\opn\depth{depth} \opn\grade{grade} \opn\height{height}
\opn\embdim{emb\,dim} \opn\codim{codim}

\opn\Tr{Tr} \opn\bigrank{big\,rank}
\opn\superheight{superheight}\opn\lcm{lcm}
\opn\trdeg{tr\,deg}
\opn\reg{reg} \opn\lreg{lreg} \opn\ini{in} \opn\lpd{lpd}
\opn\size{size}\opn\bigsize{bigsize}
\opn\cosize{cosize}\opn\bigcosize{bigcosize}
\opn\sdepth{sdepth}\opn\sreg{sreg}
\opn\link{link}\opn\fdepth{fdepth}
\opn\index{index}
\opn\index{index}
\opn\indeg{indeg}
\opn\N{N}
\opn\SSC{SSC}
\opn\SC{SC}
\opn\lk{lk}
\opn\div{div} \opn\Div{Div} \opn\cl{cl} \opn\Cl{Cl}
\opn\Spec{Spec} \opn\Supp{Supp} \opn\supp{supp} \opn\Sing{Sing}
\opn\Ass{Ass} \opn\Min{Min}\opn\Mon{Mon} \opn\dstab{dstab} \opn\astab{astab}
\opn\Syz{Syz}
\opn\reg{reg}
\opn\Ann{Ann} \opn\Rad{Rad} \opn\Soc{Soc}
\opn\Im{Im} \opn\Ker{Ker} \opn\Coker{Coker} \opn\Am{Am}
\opn\Hom{Hom} \opn\Tor{Tor} \opn\Ext{Ext} \opn\End{End}\opn\Der{Der}
\opn\Aut{Aut} \opn\id{id}

\opn\nat{nat}
\opn\pff{pf}
\opn\Pf{Pf} \opn\GL{GL} \opn\SL{SL} \opn\mod{mod} \opn\ord{ord}
\opn\Gin{Gin} \opn\Hilb{Hilb}\opn\sort{sort}
\opn\initial{init}
\opn\ende{end}
\opn\height{height}
\opn\type{type}
\opn\aff{aff} \opn\con{conv} \opn\relint{relint} \opn\st{st}
\opn\lk{lk} \opn\cn{cn} \opn\core{core} \opn\vol{vol}
\opn\link{link} \opn\star{star}\opn\lex{lex}
\opn\gr{gr}

\def\pot#1#2{#1[\kern-0.28ex[#2]\kern-0.28ex]}

\opn\dirlim{\underrightarrow{\lim}}
\opn\inivlim{\underleftarrow{\lim}}
\let\union=\cup
\let\sect=\cap
\let\dirsum=\oplus
\let\tensor=\otimes
\let\iso=\cong
\let\Union=\bigcup

\let\to=\rightarrow

\def\Implies{\ifmmode\Longrightarrow \else
        \unskip${}\Longrightarrow{}$\ignorespaces\fi}
\def\implies{\ifmmode\Rightarrow \else
        \unskip${}\Rightarrow{}$\ignorespaces\fi}
\def\iff{\ifmmode\Longleftrightarrow \else
        \unskip${}\Longleftrightarrow{}$\ignorespaces\fi}

\let\:=\colon
\newtheorem{Theorem}{Theorem}[section]
 \newtheorem{Lemma}[Theorem]{Lemma}
 \newtheorem{Corollary}[Theorem]{Corollary}
 \newtheorem{Proposition}[Theorem]{Proposition}

 \newtheorem{Example}[Theorem]{Example}
 \newtheorem{Examples}[Theorem]{Examples}

\let\epsilon\varepsilon
\let\kappa=\varkappa
\def\qed{\ifhmode\textqed\fi
      \ifmmode\ifinner\quad\qedsymbol\else\dispqed\fi\fi}
\def\textqed{\unskip\nobreak\penalty50
       \hskip2em\hbox{}\nobreak\hfil\qedsymbol
       \parfillskip=0pt \finalhyphendemerits=0}
\def\dispqed{\rlap{\qquad\qedsymbol}}

\opn\dis{dis}
\def\pnt{{\raise0.5mm\hbox{\large\bf.}}}

\opn\Lex{Lex}

\begin{document}

\title{Algebraically rigid simplicial complexes and graphs}
\author{Klaus Altmann, Mina Bigdeli, J\"urgen Herzog and Dancheng Lu}

\address{Klaus Altmann, Institut f\"ur Mathematik, Freie Universit\"at Berlin,
Arnimallee 3, D-14195 Berlin, Germany}
\email{altmann@math.fu-berlin.de}

\address{Mina Bigdeli, Department of Mathematics,  Institute for Advanced Studies in Basic Sciences (IASBS), Gava zang,
45195-1159 Zanjan, 45137-66731 Iran}
\email{m.bigdelie@iasbs.ac.ir}


\address{J\"urgen Herzog, Fakult\"at f\"ur Mathematik, Universit\"at Duisburg-Essen,
45117 Essen, Germany}
\email{juergen.herzog@gmail.com}

\address{Dancheng Lu, Department of Mathematics, Soochow University, 215006 Suzhou, P.R.China}
\email{ludancheng@suda.edu.cn}



\thanks{The paper was written while the second and the fourth authors were visiting the Department of
Mathematics of University Duisburg-Essen. They want to express their thanks for the hospitality.}
\thanks{Corresponding author: Dancheng Lu}

\begin{abstract}
We call a  simplicial complex algebraically rigid if its Stanley-Reisner ring admits no nontrivial infinitesimal deformations, and call it inseparable if it does not allow any deformation  to other simplicial complexes. Algebraically rigid   simplicial complexes are inseparable. In this paper we study inseparability and rigidity of Stanley-Reisner rings, and apply the general theory to letterplace ideals as well as to edge ideals of graphs. Classes of algebraically rigid simplicial complexes and graphs are identified.
\end{abstract}

\maketitle

\section*{Introduction}

In the study of monomial ideals it is a popular technique to polarize in order to obtain squarefree monomial ideals.  Let $S=K[x_1,\ldots,x_n]$ be the polynomial ring in $n$ indeterminates over the field $K$. Given a monomial ideal $I\subseteq S$, the polarized ideal $I^\wp$ of $I$ is a squarefree monomial ideal defined in a larger polynomial ring $S^\wp$ and $S/I$ is obtained from $S^\wp/I^\wp$ by reduction modulo a regular sequence  of linear forms consisting of differences of variables, see \cite[page 19]{HHBook} for details.  In other words, $S^\wp/I^\wp$ may be viewed  an unobstructed deformation of $S/I$ over a suitable  affine space.

 The natural question arises whether $S^\wp/I^\wp$ or any other $K$-algebra defined by a squarefree monomial ideal admits further unobstructed deformations, or at least non-trivial infinitesimal deformations. This may  be indeed the case as the third author learned from Fl{\o}ystad.
\medskip

{\it\noindent Separation.} Let $I\subseteq S$ be  a squarefree  monomial ideal, and let $y$ be an indeterminate over $S$. Fl{\o}ystad (see \cite{FGH}) calls a monomial  ideal $J\subseteq S[y]$  a {\em separation} of $I$ for the variable $x_i$ if  the following conditions hold:
\begin{enumerate}
\item[(i)] the ideal $I$ is the image of $J$ under  the $K$-algebra homomorphism $S[y]\to S$ with $y\mapsto x_i$ and $x_j\mapsto x_j $ for all $1\leq j\leq n$;

\item[(ii)] $x_i$ and $y$ divide some minimal generators of $J$;

\item[(iii)] $y-x_i$ is a non-zero divisor of $S[y]/J$.
\end{enumerate}
The ideal $I$ is called {\em separable}  if it admits a separation, otherwise it is called {\em inseparable}.

If $J\subseteq S[y]$ is a separation of $I$ for the variable $x_i$, then $K[t]\to S[y]/J$, $t\mapsto y-x_i$,  is a non-trivial unobstructed deformation of $S/I$ over $\AA^1=\Spec K[t]$. The simplest example of a squarefree monomial ideal which admits a separation is the ideal $I=(x_1x_2,x_1x_3,x_2x_3)$. One possible separation of $I$ is the ideal $J=(x_1y,x_1x_3,x_2x_3)$.

\medskip
{\it\noindent Infinitesimal deformations.}  But even if $I\subseteq S$ is inseparable, it still may admit infinitesimal deformations. We denote by $\epsilon$ a non-zero term with $\epsilon^2=0$. Let $J\subseteq S[\epsilon]$ be  an ideal. Then $S[\epsilon]/J$ is called an {\em infinitisimal deformation} of $S/I$ if the canonical $K$-algebra homomorphism $K[\epsilon]\to S[\epsilon]/J$ is flat and if $S/I$ is obtained from  $S[\epsilon]/J$ by reduction modulo $\epsilon$.  Thus if $I=(f_1,\ldots,f_m)$, then $J=(f_1+g_1\epsilon,\ldots, f_m+g_m\epsilon)$ and $K[\epsilon]\to S[\epsilon]/J$ is flat if and only if $\varphi\: I\to S/I$  with  $f_i\mapsto g_i+I$ is a
well-defined $S$-module homomorphism. In other words, the infinitesimal deformations of $S/I$ are in bijection to the elements of $I^*=\Hom_S(I,S/I)$.

\medskip
{\it\noindent Cotangent functors.} Recall that a $K$-linear map $\partial\: S\to S$ is called a {\em $K$-derivation} if $\partial fg=f\partial g+g\partial f$. The set of $K$-derivations has a natural structure as an $S$-module, and is denoted $\Der_K(S)$. In fact, $\Der_K(S)$ is a free $S$-module whose basis consists of the partial derivatives $\partial/\partial x_i$.  The infinitesimal deformation $S[\epsilon]/J$ of $S/I$ is called {\em trivial} if there exists a $K$-algebra  automorphism $S[\epsilon]\to S[\epsilon]$ which maps  $J$ to $IS[\epsilon]$. This is the case, if and only if there exists $\partial\in \Der_K(S)$ such that   $J=(f_1+\partial f_1\epsilon, \ldots, f_m+\partial f_m\epsilon)$. Thus if we consider the natural map $\delta^*\: \Der_K(S)\to I^*$ which assigns to $\partial \in \Der_K(S)$ the element $\delta^*(\partial)$ with $\delta^*(\partial)(f_i)=\partial f_i+I$, then  the non-zero elements of  $\Coker \delta^*$ are in bijection to the isomorphism classes  of non-trivial infinitesimal deformations of $S/I$, see for example \cite[Lemma 2.4]{H}.  This cokernel is denoted by $T^1(S/I)$ and is called the {\em first cotangent functor} of $S/I$.  The cotangent functors have been first introduced by Lichtenbaum and Schlessinger \cite{LS}.  The $K$-algebra $S/I$ is called {\em rigid} if it does not admit any non-trivial infinitesimal deformations. Hence $S/I$ is rigid if and only if $T^1(S/I)=0$. For simplicity we call $I$  rigid if $S/I$ is rigid. The simplest example of a squarefree monomial  ideal which is inseparable but not rigid, is the ideal $I=(x_1x_2,x_2x_3,x_3x_4)$,  see Proposition~\ref{sufficient}(a) and Theorem~\ref{chordal}. We recommend   the reader to  consult \cite[Section 3]{SBook} for general basic facts about deformation theory.

\medskip
{\it \noindent Rigidity of simplicial complexes.}  In the case that $I\subseteq S$ is a monomial ideal, $T^1(S/I)$ is a $\ZZ^n$-graded $S$-module. If moreover $I$ is a squarefree monomial ideal, then  the $\ZZ^n$-graded components of $T^1(S/I)$ have a  combinatorial interpretation as was shown by the first author and Christophersen in \cite{AC1} and \cite{AC2}. We recall some of these results in Section 1 of this paper because they are crucial for later  applications.

 Recall that a simplicial complex $\Delta$ on the vertex set $V$ is a collection of subset of $V$ such that whenever $F\in \Delta$ and $G\subset F$, then $G\in \Delta$. Frequently we denote the vertex set of $\Delta$ by $V(\Delta)$.

Since $I$ is a squarefree monomial ideal, there exists a unique simplicial complex $\Delta$ such that $I=I_\Delta$, where $I_\Delta$ is the Stanley--Reisner ideal of $\Delta$. As usual, $S/I_\Delta$ is denoted by $K[\Delta]$ and is called the Stanley--Reisner ring of $\Delta$. For simplicity we will write $T^1(\Delta)$ for $T^1(K[\Delta])$. We  say that $\Delta$ is {\em algebraically rigid} (with respect to $K$) if $K[\Delta]$ is rigid. Actually it is shown in Section~\ref{rigidity}  that algebraic rigidity of simplicial complexes does not depend on $K$.  There is the concept of rigid simplicial complexes, meaning that the simplicial complex does not admit any non-trivial automorphism. In this paper rigidity means algebraic rigidity, and will simply say that $\Delta$ is rigid if it is algebraically rigid.

\medskip
{\it \noindent Description of contents.} The aim of this paper is to characterize  rigid simplicial complexes in combinatorial terms and exhibit classes of them. By what we said before,  it follows that $\Delta$ is rigid if  and only if $T^1(\Delta)_\cb=0$ for all $\cb\in \ZZ^n$. The  important facts, shown in \cite{AC1}, regarding the graded components of $T^1(\Delta)$ that will be used throughout the paper, are the following: write $\cb=\ab-\bb$ with $\ab,\bb\in\NN^{n}$ and $\supp \ab\sect \supp \bb=\emptyset$. Here for a vector $\ab$, $\supp \ab$ is defined to be the set $\{i\in [n]\:\; a_i\neq 0\}$ where the $a_i$ are the components of $\ab$. Then
\begin{enumerate}
\item[(i)]  $T^1(\Delta)_{\ab-\bb}=0$ if $\bb\not\in \{0,1\}^{n}$, and if $\bb\in \{0,1\}^{n}$, then $T^1(\Delta)_{\ab-\bb}$ depends only on $\supp\ab$ and $\supp\bb$;
\item[(ii)] $T^1(\Delta)_{\ab-\bb}=T^1(\link_\Delta \supp \ab)_{-\bb}$.
\end{enumerate}
We say that $\Delta$ is {\em $\emptyset$-rigid}  if  $T^1(\Delta)_{-\bb}=0$ for all $\bb\in \{0,1\}^n$. Thus, by (ii), $\Delta$ is rigid, if and only if all its links are $\emptyset$-rigid. These and other facts are recalled in Section~\ref{rigidity}. We close the section by applying the general theory to characterize inseparable simplicial complexes. Say, $\Delta$ is a simplicial complex on the vertex set $[n]$. To each vertex $i$ of $\Delta$ one attaches a graph $G_{\{i\}}(\Delta)$ whose vertices are those faces $F\in \Delta$ for which $F\union\{i\}\not\in  \Delta$. The edges of $G_{\{i\}}(\Delta)$  are those $\{F,G\}$ for which  $F\subsetneq G$ or $G\subsetneq F$. In Theorem~\ref{insep} we show that $\Delta$ is inseparable if and only if $G_{\{i\}}(\Delta)$ is connected for $i=1,\ldots, n$, and that this is equivalent to the condition that $T^1(\Delta)_{-\eb_i}=0$ for $i=1,\ldots,n$. Here $\eb_i$ denotes the $i$th canonical basis vector of $\ZZ^n$.

Let $\dim_K T^1(\Delta)_{-\eb_i}=k$. Iterating simple separation steps one can construct a simplicial complex $\widetilde{\Delta}$  on the vertex set $([n]\setminus\{i\})\union \{v_0,v_1,\ldots,v_k\}$ defined as $k$-separation of $\Delta$ for the vertex $i$ having the  property that $T^1(\widetilde{\Delta})_{-\bb}=0$ for all $\bb$ with $\supp \bb\subseteq \{v_0,\ldots,v_k\}$,  and such that $K[\Delta]$ is obtained from $K[\widetilde{\Delta}]$ by cutting down by a regular sequence consisting of differences of variables.

In Section ~\ref{joinsetc} we consider various operations on simplicial complexes and study their behaviour with respect to rigidity. In Proposition~\ref{join} it is shown that the join $\Delta_1*\Delta_2$ of the simplicial complexes $\Delta_1$ and $\Delta_2$ is rigid if and only if this is the case for $\Delta_1$ and $\Delta_2$.

More complicated is the situation for the disjoint union of two simplicial complexes $\Delta_1$ and $\Delta_2$. Here we assume that none of the two simplicial complexes is the empty set and that their $0$-dimensional faces correspond to their vertex set, a condition that we do not require  in general. Under these (very weak) assumptions it is shown in Theorem~\ref{union} that $\Delta_1\union \Delta_2$ is inseparable  if and only if $\Delta_1$ and $\Delta_2$ are simplices, and that $\Delta_1\union \Delta_2$ is rigid if and only in addition one  of the simplices has positive dimension. As a consequence we see that a disconnected simplicial complex of positive dimension is never rigid, unless all its components are simplices.

Finally in Theorem~\ref{circ} we consider what we call the {\em circ} of two simplicial complexes, denoted by $\Delta_1\circ \Delta_2$. Suppose that  $V_i$ is the vertex set of $\Delta_i$ and that $V_1\sect V_2=\emptyset$. Then, by definition,  $F\subseteq V_1\union V_2$ is a face of  $\Delta_1\circ \Delta_2$ if and only if  either   $F\cap V_1$ is a face of $\Delta_1$ or  $F\cap V_2$ is a face of $\Delta_2$. Note that $\Delta_1\circ \Delta_2=(\Delta_1*\langle V_2\rangle)\union (\langle V_1\rangle*\Delta_2)$ and that $I_{\Delta_1\circ \Delta_2}=(I_{\Delta_1}I_{\Delta_2})$.
It turns out that $\Delta_1$ and $\Delta_2$ are rigid if $\Delta_1\circ \Delta_2$ is rigid and $I_{\Delta_1}, I_{\Delta_2}\neq 0$.  The  converse is true only under some additional assumptions.

A motivation for us to study  the circ-operation resulted  from the desire to classify the rigid letterplace ideals, see \cite{FGH}. Given two finite posets $\MP$ and $\MQ$, one assigns a monomial ideal $L(\MP,\MQ)$, which in the case $\MP=[n]$ or $\MQ=[n]$  is called a letterplace ideal or a co-letterplace ideal, respectively. Letterplace  and  co-letterplace  ideals have been considered before
in \cite{EHM}. In the paper \cite{FGH} it is shown that all letterplace ideals are inseparable. Here we show that $L(\MP,\MQ)$ is rigid if and only if no two elements of $\MP$ are comparable, see Theorem~\ref{dan}. In the proof of one direction of this theorem we need the circ-construction.

The last section of this paper is concerned with the rigidity of edge ideals. Given a finite simple graph on the vertex set $[n]$ one assigns  to it the so-called edge ideal $I(G)$ generated by the monomials $x_ix_j$ with $\{i,j\}$ an edge of $G$. Obviously $I(G)=I_{\Delta(G)}$ for some simplicial complex $\Delta(G)$. This simplicial complex is called the {\em independence complex} of $G$. Indeed, its faces are the independent sets of $G$, that is, the subsets of $[n]$  which do not contain any edge of $G$. We say that $G$ is rigid if $\Delta(G)$ is rigid. Again there exist already various concepts of rigid graphs which should not be confused with the definition of rigidity used in this paper. Similarly, we say that $G$ is inseparable if $I(G)$ is inseparable. The ultimate goal would be to classify all rigid and inseparable graphs. It is not clear whether a nice description of these classes of  graphs is possible. However with some additional assumptions on the graphs, inseparable or rigid graphs can be  characterized combinatorially.
Recall that a vertex $i$ of $G$ is called a {\em free vertex} if it belongs to only one edge, and an edge is called a {\em leaf} if it has a free vertex. Finally an  edge $e$ of $G$ is called a {\em branch}, if there exists a leaf $e'$  with $e'\neq e$ such that $e\sect e'\neq \emptyset$. Our main result on rigidity of graphs is formulated in Theorem~\ref{chordal}: Let $G$ be a   graph on the vertex set $[n]$ such that  $G$ does not contain any induced cycle of length $4$, $5$ or $6$. Then $G$ is   rigid  if and only if each edge of $G$ is a branch and each vertex of a $3$-cycle of $G$  belongs to a leaf. Theorem~\ref{chordal}  has several  consequences. In Corollary~\ref{againchordal} it is shown that a chordal graph $G$ is   rigid if and only if each edge of $G$ is a branch and each vertex of a $3$-cycle of $G$  belongs to a leaf. Another consequence is the fact that a graph with the property that all cycles have length $\geq 7$ is rigid if and only if each of its edges is a branch, see Corollary~\ref{forest}. This result implies in particular that a forest consisting only of branches is rigid.   Finally we notice in Corollary~\ref{rigidcycle} that a cycle is rigid if and only if it is a $4$- or $6$-cycle.

\section{The cotangent functor $T^1$ and  rigid and inseparable  Stanley--Reisner rings}
\label{rigidity}

\noindent
{\it The cotangent functor $T^1$}.  Let $\Delta$ be a simplicial complex on the vertex set $V(\Delta)=[n]$ where $[n]=\{1,\ldots,n\}$.
We denote by $[\Delta]$ the set of elements $i\in [n]$ with $\{i\}\in \Delta$. Let $F_1,\ldots, F_m\subseteq [n]$. We denote by $\langle F_1,\ldots, F_m\rangle$ the smallest simplicial complex $\Delta$ with $F_i\in \Delta$ for $i=1,\dots, m$. The elements of $\Delta$ are called {\em faces}. A {\em facet} of $\Delta$ is a face of $\Delta$ which is maximal with respect to inclusion. The set of facets of $\Delta$ will be denoted by $\MF(\Delta)$.

 We fix a field $K$. The ideal $I_\Delta$ denotes the Stanley-Reisner ideal in $S=K[x_1,\ldots,x_n]$, that is, the ideal generated by the monomial $x_N$ with $N\subseteq [n]$ a non-face of $\Delta$. Here $x_N=\prod_{i\in N}x_i$. The $K$-algebra $K[\Delta]=S/I_\Delta$ is called the Stanley-Reisner ring of $\Delta$.

The cotangent cohomology modules $ T^i(K[\Delta])$ which we denote by
$T^i(\Delta)$ are $\ZZ^{n}$-graded.  We quote several facts about the $\ZZ^n$-graded components of $T^i(\Delta)$ which were shown in \cite{AC1}.

\medskip
We write $\cb\in \ZZ^{n}$ as $\ab-\bb$ with $\ab,\bb\in\NN^{n}$ and $\supp \ab\sect \supp \bb=\emptyset$, and  set  $A=\supp \ab$ and $B=\supp \bb$. Here $\NN$ denotes the set of non-negative integers, and as in the introduction the support of a vector $\ab \in \NN^n$ is defined to be the set $\supp \ab=\{i\in [n]\:\; a_i\neq 0\}$.

\begin{Theorem}[\cite{AC1}, Theorem 9]
\label{theorem9}
{\em (a)} $T^i(\Delta)_{\ab-\bb}=0$ if $\bb\not\in \{0,1\}^{n}$.

{\em (b)} Assuming $\bb\in \{0,1\}^{n}$, $T^i(\Delta)_{\ab-\bb}$ depends only  on $A$ and $B$.
\end{Theorem}

Recall that for a subset $A$ of $[n]$, the {\em link} of $A$ is defined to be
\[
\link_\Delta A=\{F\in \Delta\:\; F\sect A=\emptyset, \; F\union A\in \Delta\}
\]
with vertex set $V(\link_\Delta A)=[n]\setminus A$.

\medskip
We will also need the following result:

\begin{Proposition}[\cite{AC1}, Proposition 11]
\label{Proposition11}
{\em (a)} $T^i(\Delta)_{\ab-\bb}=0$,  unless
\[
A\in \Delta\quad \mathrm{and} \quad  \emptyset\neq B\subseteq [\link_{\Delta}A].
\]

{\em (b)}  $T^i(\Delta)_{\ab-\bb}=T^i(\link_\Delta A)_{-\bb}$.
\end{Proposition}

In the present paper, we are only interested in $T^1$. Because of Proposition~\ref{Proposition11}(b) it is important to know how to compute $T^1(\Delta)_{-\bb}$ for $B\subseteq {[\Delta]}$. For this purpose we introduce some notation.

Let  $\MY$  be a collection of subsets of $[n]$. We set  $\MK^0(\MY)=\{\lambda:\MY\to K\}$ and
$$
\MK^1(\MY)=\big\{\lambda:\{(Y_0,Y_1)\in \MY^2:\  Y_0\cup Y_1\in \MY\}\to K\big\}
$$
and define the $K$-linear map $d:\MK^0(\MY)\to \MK^1(\MY)$  by $(d\lambda)(Y_0,Y_1)=\lambda(Y_1)-\lambda(Y_0)$.

Next given $B\subseteq [n]$ and $\Delta$, we define
\begin{eqnarray*}
N_B(\Delta)&=&\{F\in \Delta :\  F\cap B=\emptyset,\; F\cup B\notin \Delta\},\\
\widetilde{N}_B(\Delta)&=&\{F\in N_B(\Delta):  \mbox{there exists} \ B'\subsetneq B
\mbox{ with }F\cup B'\notin \Delta\}.
\end{eqnarray*}

With the notation introduced one has
\begin{Proposition}[\cite{AC1}, Corollary 6]
\label{compute}
{\em (a)} Suppose $|B|\geq 2$. Then
$$T^1(\Delta)_{-\bb}=\Ker\big (\MK^0(N_B(\Delta))\stackrel{(d,r)}{\longrightarrow} \MK^1(N_B(\Delta))\oplus \MK^0(\widetilde{N}_B(\Delta))\big),$$
where $d:\MK^0(N_B(\Delta))\to \MK^1(N_B(\Delta))$ is the map as defined above and $r:\MK^0(N_B(\Delta))\to \MK^0(\widetilde{N}_B(\Delta))$ is the restriction map.

{\em (b)} For $|B|=1$,
the $K$-dimension of $T^1(\Delta)_{-\bb}$ is one less than the $K$-dimension of the kernel given in {\em (a)}.
\end{Proposition}

\bigskip
\noindent
{\it Rigidity}.
The simplicial complex $\Delta$ is called {\em $\emptyset$-rigid}  if  $T^1(\Delta)_{-\bb}=0$ for all $\bb\in \{0,1\}^n$.
For $\emptyset$-rigidity it is enough to check the vanishing of $T^1(\Delta)_{-\bb}$ for  $\bb\in \{0,1\}^n$ with  $\supp \bb\subseteq  [\Delta]$. It follows from Proposition~\ref{Proposition11} that $\Delta$ is rigid if and only if $\link_{\Delta}A$ is $\emptyset$-rigid  for all $A\in \Delta$.  Thus we will assume $\ab=0$ from now on.

As an immediate consequence of Proposition~\ref{compute} one obtains

\begin{Corollary}
\label{useful}
{\em (a)} Suppose that $|B|\geq 2$. Then
$
T^1(\Delta)_{-\bb}=\Lambda_{B}(\Delta)$, where
$$\Lambda_{B}(\Delta)=\{\lambda:N_B(\Delta)\to K \: \;
\lambda |_{\widetilde{N}_B(\Delta)}=0 \; \mathrm{and} \;
\lambda(F)=\lambda(G)\; \mathrm{ whenever  }\; F\subseteq G\}.
$$

{\em (b)}  $\dim_KT^1(\Delta)_{-\bb}=\dim_K\Lambda_{B}(\Delta)-1$ if $|B|=1$.
\end{Corollary}

\medskip
 Let $B$ be a subset of $[n]$. We define $G_B(\Delta)$ to be the graph whose vertex set is $N_B(\Delta)$ and for which $\{F,G\}$ is an edge of $G_{B}(\Delta)$  if and only if $F\subsetneq G$ or $G\subsetneq F$. It follows that $\lambda\in \Lambda_{B}(\Delta)$ is constant on the connected components of $G_{B}(\Delta)$. Note that if $|B|=1$, then $\widetilde{N}_B(\Delta)=\emptyset$.
\medskip
We see that if $|B|\geq 2$, then
\begin{eqnarray}
\label{number1}
\dim_KT^1(\Delta)_{-\bb}&=&\text{number of connected components of $G_{B}(\Delta)$}\\
&& \text{ which contain no element of $\widetilde{N}_B(\Delta)$,}\nonumber
\end{eqnarray}
and if $B=\{i\}$, then
\begin{eqnarray}
\label{number2}
\dim_KT^1(\Delta)_{-\eb_i}=\text{number of connected components of $G_{\{i\}}(\Delta)$}-1
\end{eqnarray}

Hence the rigidity of a simplicial complex is independent of the field $K$.

\begin{Examples}{\em
(a) Let $2^{[n]}=\langle [n]\rangle$ be the simplex on the vertex set $[n]$. For each $B\subseteq [n]$
we have $\widetilde{N}_B(2^{[n]})=\emptyset$.
This implies that  $2^{[n]}$ is $\emptyset$-rigid. Moreover,
for each $A\in 2^{[n]}$, its link is a  simplex, too.
Thus, $2^{[n]}$ is rigid.
Of course, this is known before because  $K[2^{[n]}]=K[x_1,\ldots,x_n]$.

(b) Fix $n\geq 2$,  and let  $\Gamma=2^{[n]}\setminus\{[n]\}$ be the boundary of $2^{[n]}$.  Then
$N_{[n]}(\Gamma)=\{\emptyset\}$, but $\widetilde{N}_{[n]}(\Gamma)=\emptyset$. In particular,
$\dim_K T^1(\Gamma)_{-\bb}=1$ for $\bb=(1,\ldots,1)$. Again this follows also directly  from the fact that $K[\Gamma]=K[x_1,\ldots,x_n]/(x_1x_2\cdots x_n)$.

(c) let  $\Delta=\langle \{1\},\ldots,\{n\}\rangle$. Then $\Delta$   is $0$-dimensional. The set
$B=\{1\}$ yields $N_{B}(\Delta)=\{\{2\},\ldots,\{n\}\}$ and
$\widetilde{N}_{B}(\Delta)=\emptyset$. Hence,
$T^1(\Delta)_{-\bb}\neq 0$ for $n\geq 3$, and so $\Delta$ is not rigid.
Note that $K[\Delta]=K[x_1,\ldots,x_n]/(x_ix_j\:\;  i\neq j)$.

The ideal $(x_ix_j\:\;  i\neq j)$ may be interpreted as the edge ideal of the complete graph on the vertex set $[n]$. Rigidity of edge ideals will be discussed in details in Section~\ref{graphs}.}
\end{Examples}

The following lemma tells us  when  $T^1(\Delta)_{-\bb}$  vanishes if  $\supp \bb\not\in \Delta$. We denote by $2^B$ the simplex on the vertex set $B$.

\begin{Lemma}
\label{Bnot}
Let $\Delta$ be a simplicial complex on the vertex set $[n]$, and let $\bb\in \{0,1\}^n$. Let $B=\supp \bb$ and assume that $B\not\in \Delta$.
\begin{enumerate}
\item[(a)] Suppose that $|B|\geq 2$. Then $T^1(\Delta)_{ -\bb}= 0$ if and only if  $\widetilde{N}_B(\Delta)\neq \emptyset$.  Moreover, $T^1(\Delta)_{-b}\neq 0$ implies that
     the boundary $2^B\setminus \{B\}$ of $2^B$ is contained in  $\Delta$.
\item[(b)]  Suppose that $|B|=1$. Then $T^1(\Delta)_{-\bb}=0$.
\end{enumerate}
\end{Lemma}

\begin{proof} Since $B\not\in \Delta$ it follows that $\emptyset \in N_B(\Delta)$. Therefore, $\lambda(F)=\lambda(\emptyset)$ for all $\lambda \in \Lambda_B(\Delta)$. Thus the $K$-vector space  $\Lambda_B(\Delta)$ is generated by one element $\lambda_0$ which is forced to be the $0$-element if $\widetilde{N}_B(\Delta)\neq \emptyset$ and which may be chosen to be the constant map with $\lambda_0(\emptyset)=1$, otherwise. This proves (a).  Also (b) follows from this considerations keeping in mind Corollary~\ref{useful}(b). Alternatively,  statement (b)  follows from Lemma~\ref{Proposition11}(a).
\end{proof}

\bigskip
\noindent
{\it Separation.}  Let, as before,  $\Delta$ be a simplicial complex on the vertex set $[n]$. We say that $\Delta$ is {\em separable}, if for some $i$, $I_\Delta$ admits a separation for $x_i$. Otherwise, we say that $\Delta$  is {\em inseparable}. In the further discussions we refer to the conditions (i), (ii) and (iii) for separation, as given  in the introduction.

Let $I=I_\Delta$ be minimally generated by the monomials $u_1,\ldots, u_m$. We first observe:
\begin{Lemma}\label{lemma1.7}
   If $J$ is a separation of $I$ for the variable $x_i$, then $T^1(\Delta)_{-\eb_i}\neq 0$.
   \end{Lemma}  \begin{proof}
    By condition (iii), $S/I$ is obtained from $S[y]/J$ by reduction modulo a linear form  which is a regular element on $S[y]/J$. This implies that $I$ and $J$ are minimally generated by the same number of generators.  Let $J$  be minimally generated by $v_1,\ldots,v_m$. We may assume that $y$ divides $v_1,\ldots,v_k$ but does not divide the other generators of $J$. We may furthermore assume that for all $i$,  $v_i$ is mapped to $u_i$ under the $K$-algebra homomorphism (i). Then we may write
\[
J=(u_1+(u_1/x_i)(y-x_i),\ldots, u_k+(u_k/x_i)(y-x_i), u_{k+1},\ldots,u_m).
\]
From this presentation and by (iii)  it follows  that $S[y]/J$ is an unobstructed deformation of $S/I$ induced by the element
$[\varphi]\in T^1(S/I)_{-\eb_i}$, where $\varphi\in I^*$ is the $S$-module homomorphism with $\varphi(u_j)=u_j/x_i+I$ for $j=1,\ldots,k$ and $\varphi(u_j)=0$, otherwise.

Condition (ii)  makes sure that $S[y]/J$  is a non-trivial deformation of $S/I$. Indeed,  suppose $[\varphi]=0$.  Observe, that $\deg \varphi=-\eb_i$. Therefore,   $\varphi\in (\Im \delta^*)_{-\eb_i}$,  which is the $K$-vector space spanned by $\varphi_i=\delta^*(\partial/\partial x_i)$.  Here $\delta^*\: \Der_K(S)\to I^*$ is the map as defined in the introduction with $\delta^*(\partial)(f)=\partial f+I$ for $\partial \in \Der_K(S)$ and $f\in I$.
It follows that $\varphi= \lambda \varphi_i$ for some $\lambda \in K$. Since $\varphi(u_1)= u_1/x_i+I=\lambda \varphi_i(u_1)$ it follows that $\lambda=1$. On the other hand, by condition (ii),  there exists $j>k$ such that $x_i|u_j$  and  $\varphi(u_j)= I\neq u_j/x_i+I=\varphi_i(u_j)$. This is a contradiction. \end{proof}

It follows from the above result  that $\Delta$ is inseparable if $T^1(\Delta)_{-\eb_i}=0$ for $i=1,\ldots,n$.
Moreover the deformation $K[y]\to  S[y]/J$  induces an infinitesimal deformation $K[\epsilon]\to  S[\epsilon]/\bar{J}$ by reduction modulo $y^2$. As explained in the introduction, this infinitesimal deformation yields an element in $T^1(\Delta)$. This assignment is called the {\it Kodaira-Spencer map}. The arguments given above, even show that the image of the deformation $K[y]\to  S[y]/J$ in $T^1(\Delta)$ via the Kodaira-Spencer map is non-trivial.

\medskip
Finally we obtain

\begin{Theorem}
\label{insep}
The following conditions are equivalent:
\begin{enumerate}
\item[(a)] $\Delta$ is inseparable.
\item[(b)]  $G_{\{i\}}(\Delta)$ is connected for $i=1,\ldots,n$.
\item[(c)] $T^1(\Delta)_{-\eb_i}=0$ for $i=1,\ldots,n$.
\end{enumerate}
\end{Theorem}

\begin{proof}
(c) \implies (a) follows from Lemma~\ref{lemma1.7}, and (b)\iff (c) follows from  the equality (\ref{number2})  in the preceding subsection.

(a)\implies (b):  Suppose $G_{\{i\}}(\Delta)$ is not connected for some $i$. Then the vertex set of $G_{\{i\}}(\Delta)$ can be written as a disjoint union  $V(G_{\{i\}}(\Delta) )=\MA\union \MB$ such that  for all $F\in \MA$ and all $G\in \MB$, neither  $F\subsetneq G$ nor $G\subsetneq F$.

Let $I=I_\Delta$,  $I(\MA)= (x_F\:\; F\in \MA)$ and $I(\MB)=(x_G\:\; G\in \MB)$. Then
\[
I=(x_i I(\MA),  x_iI(\MB), u_1,u_2,\ldots,u_t),
\]
where none of the $u_j$  is  divisible by $x_i$.

Note that if $\{x_{F_1},\ldots,x_{F_r}\}$ is the minimal set of monomial generators of $I(\MA)$ and $\{x_{G_1},\ldots,x_{G_s}\}$ is the minimal set of monomial generators of $I(\MB)$, then
\[
\{x_ix_{F_1},\ldots,x_ix_{F_r}, x_ix_{G_1},\ldots,x_ix_{G_s}\}
\]
is the set of monomials of the minimal monomial set of generators  of $I$  which are divisible by $x_i$.
Thus, the ideal
$
J=(yI(\MA), x_iI(\MB), u_1,u_2,\ldots,u_t)\subseteq S[y]$ satisfies the conditions (i) and (ii) of a separation of $I$. We will show that $y-x_i$ is a non zero-divisor of $S[y]/J$. This will then imply that $\Delta$ is separable, yielding a contradiction.

Indeed, suppose $y-x_i$  is a zero-divisor  of $S[y]/J$. Then $y-x_i$ belongs to a minimal prime ideal $P$ of $J$. Since $P$ is a monomial prime ideal it follows that $y,x_i\in P$.

Now let $F\in \MA$ and $G\in \MB$ and suppose that $F\union G\in \Delta$. Then $F\union G\in N_{\{i\}}(\Delta)$, and hence  $F\union G\in \MA$ since $F\subsetneq F\union G$, and similarly $F\union G\in \MB$ since $G\subsetneq F\union G$. This is a contradiction. Therefore, $F\union G\not\in \Delta$ for all $F\in \MA$ and $G\in \MB$. This implies that
$I(\MA) I(\MB)\subseteq I$.  It follows that $I(\MA) I(\MB)\subseteq J$, since $x_i$ does  not divide any of the generators of $I(\MA) I(\MB)$.  Now since $I(\MA) I(\MB)\subseteq P$ we conclude  that $I(\MA)\subseteq P$ or $I(\MB)\subseteq P$. As  $P$ is a minimal prime ideal of $J$,  we see  that $y\not\in P$ if  $I(\MA)\subseteq P$ and $x_i\not\in P$ if $I(\MB)\subseteq P$. In any case we obtain a contradiction.
\end{proof}

\medskip
\noindent
{\it $k$-Separation.}
Let $\MA$ and $\MB$ be two finite collections of sets with $F\sect G=\emptyset$ for all $F\in \MA$ and $G\in \MB$.  As it is common, we denote by $\MA*\MB$ the  {\em join} of $\MA$ and $\MB$, where   $\MA*\MB=\{F\union G\:\; F\in \MA, G\in \MB\}$. If  $\Gamma$ and $\Sigma$ are simplicial complexes, then  the join   $\Gamma*\Sigma$ is again a simplicial complex. The vertex set of   $\Gamma*\Sigma$ is the set $V(\Gamma)\union V(\Sigma)$.  In \cite[Lemma 4.3]{AC2}, it was shown that all $T^1_{-\eb_i}(X)$ vanish
whenever $X$ is a combinatorial manifold without boundary.
Now we will show that beyond this special case
the deformations in degree $-\eb_i$ are unobstructed. More precisely, we prove the following statement:

\begin{Proposition}  Assume that $V(\Delta)=[n]$ and $\dim_KT^1(\Delta)_{-\eb_i}=k$.  Then there exists a simplicial complex $\widetilde{\Delta}$ on the vertex set $([n]\setminus \{i\})\union \{v_0,\ldots,v_k\}$ such that $I_{\widetilde{\Delta}}$ is obtained from $I_{\Delta}$ by $k$ times separations for the variable $x_i$. Moreover $T^1(\widetilde{\Delta})_{-\bb}=0$ for any $\bb$ whose support $\mathrm{supp}\ \bb\subseteq \{v_{0},\ldots,v_k\}$.
\end{Proposition}\begin{proof}
We may decompose $\Delta$ as a disjoint union
\[
\{\emptyset, \{i\}\}*\link_\Delta\{i\}\union N_{\{i\}}(\Delta).
\]
Since  $\dim_KT^1(\Delta)_{-\eb_i}=k$, we have $N_{\{i\}}(\Delta)$ splits into $k+1$ connected components
\[
N_{\{i\}}(\Delta)=A_0\union A_1\union\ldots \union A_k.
\]
As in the proof of Theorem~\ref{insep}, we can express $I_{\Delta}$ as
\[
I_{\Delta}=(\{x_ix_F\:\; F\in \Union_{\ell=0}^k A_{\ell}\}, u_1,\ldots,u_t),
\]
where the generators $u_j$ are not divisible by $x_i$.

We define a new simplicial complex $\widetilde{\Delta}$  on the vertex set $([n]\setminus\{i\})\union \{v_0,v_1,\ldots,v_k\}$ by
\[
\widetilde{\Delta}=\Omega*\link_{\Delta}\{i\}\union \Union_{\ell=0}^{k}(\Omega_{\ell}*A_{\ell}).
\]
Here $\Omega=\langle \{v_0,\ldots,v_k\}\rangle$, and $\Omega_{\ell}=\langle \{v_0,\ldots,v_k\}\setminus \{v_{\ell}\}\rangle$ for $\ell=0,\ldots,k$.

 Set $T=K[x_1,\ldots,x_{i-1},x_{i+1},\ldots,x_n,y_0,\ldots,y_k]$.
 Then $I_{\widetilde{\Delta}}$ is an ideal of $T$ and
$$I_{\widetilde{\Delta}}=(\{y_{\ell}x_F\:\; \ell=0,\ldots, k, F\in A_{\ell}\},  u_1,\ldots,u_t).$$

\noindent
Theorem 1.7 provides the induction step of proving the $k$-separability. Thus applying induction
it can be shown that $S/I$ is isomorphic to $T/I_{\widetilde{\Delta}}$ modulo the regular sequence $y_1-y_0,\ldots,y_k-y_{k-1}$,  and furthermore $T^1(\widetilde{\Delta})_{-\bb}=0$ for any $\bb$ whose support $\mathrm{supp}\ \bb\subseteq \{v_{0},\ldots,v_k\}$.
\end{proof}

We  call  $I_{\widetilde{\Delta}}$ a {\it $k$-separation} of $I_{\Delta}$.

\section{Joins, disjoint unions and  circs of simplicial complexes}
\label{joinsetc}

In this section we consider simplicial complexes arising from pairs of simplicial complexes and study their behaviour with respect to rigidity. Part of the results will be applied to classify rigid  algebras defined by letterplace ideals.

\bigskip
\noindent
{\it Monomial localization.}
In the following localization will be one of the tools in the proofs.  Let $K$ be a field, $S=K[x_1,\ldots,x_n]$ the polynomial ring over $K$, $I\subseteq S$  a monomial ideal and $P\subseteq S$ a monomial prime ideal.  Then $P=P_F$ where $F\subseteq [n]$ and $P_F=(x_i\: i\in F)$.

 We observe that if $(S/I)_P$ denotes ordinary localization of $S/I$ with respect to the prime ideal $P$, then  $(S/I)_P=S_P/I(P)S_P$, where $I(P)\subseteq S(P):=K[x_i\:\; i\in F]$ is the monomial ideal which is obtained from $I$ by the substitution $x_i\mapsto 1$ for $i\not \in F$. The ideal $I(P)$ is called the {\em monomial localization} of $I$ with respect to $P$.

\begin{Lemma}
\label{local}
Let $I\subseteq S$ be a monomial ideal, $P\subseteq S$ a monomial prime ideal, and $F\subseteq [n]$.
\begin{enumerate}
\item[(a)] $T^1(S/I)_P\iso (T^1(S(P)/I(P))[x_i\:\; i\not\in P])_P$.
\item[(b)] Suppose that  $T^1(S(P)/I(P))\neq 0$.  Then $T^1(S/I)\neq 0$.
\item[(c)] Let $I$ be a squarefree monomial ideal. Let $\Delta$ be  the simplicial complex  on the vertex set $[n]$ with  $I_{\Delta}=I$. Then  $S(P_{\overline{F}})/I_{\Delta}(P_{\overline{F}})=S(P_{\overline{F}})/I_\Gamma$, where $\overline{F}=[n]\setminus F$ and $\Gamma=\link_{\Delta}(F)$. In particular, $\Gamma$ is rigid if $\Delta$ is rigid.
\end{enumerate}
\end{Lemma}

\begin{proof}
(a)  Note that $\Der_K(S)=\Hom_S(\Omega_{S/K},S)$, where $\Omega_{S/K}$ is module of differentials of $S/K$, see \cite[Definition, p. 384]{Ei}. Since $\Omega_{S/K}$ localizes (see \cite[Proposition 16.6]{Ei}, the same holds true for $\Der_K(S)$. In other words, $\Der_K(S)_P=\Der_K(S_P)$. From this fact one easily deduces that $T^1(S/I)$ localizes, that is $T^1(S/I)_P=T^1((S/I)_P)$. Therefore,
\begin{eqnarray*}
T^1({S}/{I})_P&\iso& T^1(({S}/{I})_P)\iso T^1(({S}/{I(P)}S)_P)\iso  T^1((S(P)/I(P))[x_i:\ x_i\notin P])_P\\
&\iso &(T^1(S(P)/I(P))[x_i\:\; i\not\in P])_P.
\end{eqnarray*}
The last isomorphism follows from the fact that $T^1(R[y])\iso T^1(R)\tensor_RR[y]=T^1(R)[y]$ for a polynomial extension $R\to R[y]$.

(b) Suppose that $T^1(S(P)/I(P))\neq 0$ and let $\mm_P$ be the graded maximal ideal of $S(P)$ . Then $T^1(S(P)/I(P))_{\mm_P}\neq 0$ because $\mm_P\in \Supp(T^1(S(P)/I(P)))$. It follows that
$(T^1(S(P)/I(P))[x_i\:\; i\not\in P])_P\neq 0$, since $\mm_PS=P$. Hence the assertion follows from part (a).

(c) Since $\mathcal{F}(\Gamma)=\{G\setminus F\:\; G\in \MF(\Delta), F\subseteq G\}$, we obtain that $I_\Gamma\subseteq S(P_{\overline{F}})$ is given by
\[
I_{\Gamma}=\bigcap_{G\in \mathcal{F}(\Delta),F\subseteq G} P_{\overline {F}\setminus (G\setminus F)}=\bigcap_{G\in \mathcal{F}(\Delta),F\subseteq G} P_{\overline{G}}.
\]
On the other hand,
\[
I_{\Delta}(P_{\overline{F}})=\bigcap_{G\in \mathcal{F}(\Delta)}P_{\overline{G}}(P_{\overline{F}})=\bigcap_{G\in \mathcal{F}(\Delta),F\subseteq G} P_{\overline{G}}.
\]
Hence $I_{\Gamma}=I_{\Delta}(P_{\overline{F}})$.
\end{proof}

Note that the fact stated in Lemma~\ref{local}(c) which says that each link of a rigid simplicial complex is again rigid can also be deduced from  Proposition~\ref{Proposition11}(b).

\medskip

\begin{Example}{\em Let $I$ be a squarefree monomial ideal of $K[x_1,\ldots,x_n]$ generated in degree $(n-1)$.

(a) If $n=3$, then $I$ is rigid if and only if $I$ is generated by two monomials.

(b) If $n\geq 4$, then $I$ is not rigid.}
\end{Example}

\begin{proof} In the following we may assume that $I$ is not a principal ideal, because in this case $I$ is a complete intersection generated in degree $\ge 2$, and hence $I$ is not rigid.

(a) For $n=3$, $I$ is an edge  ideal of a nonempty simple graph with 3 vertices. Edge ideals will be treated in detail on Section~\ref{graphs}. A nonempty simple graph with 3 vertices is an isolated edge or a path of length 2 or a triangle.  It follows from Theorem~\ref{chordal} that in those graph, only the second is rigid. This proves our result.

(b) We proceed by induction on $n$.  Assume $n=4$. If $I$ is generated by 4 monomials, then the monomial localization $I(P)$ of $I$ with respect to $P=(x_1,x_2,x_3)$ is the edge ideal of a triangle, and so $I(P)$ is not rigid because of Theorem~\ref{chordal}. By Lemma~\ref{local}(b) it follows that $I$ is not rigid.  If $I$ is generated by 3 monomials, then $I$ is the product of a variable, say $x_1$,  and the  edge ideal $(x_2x_3, x_2x_4,x_3x_4)$ of a triangle, and so it is not rigid. Indeed, this follows again from Lemma~\ref{local}(b) by monomial localization with respect to $P=(x_2,x_3,x_4)$. If $I$ is generated by 2 monomials, then it is the product of a monomial of degree 2  and an ideal generated by variables, and again $I$ is not rigid. Assume $n>4$. If $I$ is generated by $n$ monomials, then the monomial localization $I(P)$ of $I$ with respect to $P=(x_1,x_2,x_3)$ is the edge ideal of a triangle, and so $I$ is not rigid. If $I$ is generated by less than $n$ monomials, then $I$ is the product of a variable and an ideal $J$ which is generated  in degree $n-2$ in the remaining $n-1$ variables. It follows from the induction hypothesis that $J$ is not rigid and so $I$ is not rigid.
\end{proof}

\medskip
\noindent
{\it  Joins.} Note that the notation ``join" has  been defined in the first section. Thus,  if  $\Delta_1$ and $\Delta_2$ are simplicial complexes on disjoint vertex sets, then  the join   $\Delta_1*\Delta_2$ is  a simplicial complex on the vertex set $V(\Delta_1)\union V(\Delta_2)$.

In the remaining part of this subsection we always assume that
 $\Delta_1$ and $\Delta_2$ are simplicial complexes with disjoint vertex sets $E_1$ and $E_2$ respectively.
\begin{Proposition}\label{join}  Let $\mathbf{a}_i\in \mathbb{N}^{E_1\cup E_2}$ and $\mathbf{b}_i\in \{0,1\}^{E_1\cup E_2}$ for $i=1,2$ such that $\mathrm{supp}(\mathbf{a}_i)\cap \mathrm{supp}(\mathbf{b}_i)=\emptyset$,  $\mathrm{supp}(\mathbf{a}_i)\cup \mathrm{supp} (\mathbf{b}_i)\subseteq [\Delta_i]$.  Then
\[T^1(\Delta_1\ast \Delta_2)_{\mathbf{a}_1+\mathbf{a}_2-\mathbf{b}_1-\mathbf{b}_2}=\left\{
                                                                                   \begin{array}{ll}
                                                                                     0, & \hbox{$\mathbf{b}_1\neq 0$ and $\mathbf{b}_2\neq 0$;} \\
                                                                                     0, & \hbox{either $\mathrm{supp}(\mathbf{a}_1)\notin \Delta_1$ or $\mathrm{supp}(\mathbf{a}_2)\notin \Delta_2$;} \\
                                                                                    T^1(\Delta_1)_{\mathbf{a}_1-\mathbf{b}_1} , & \hbox{$\mathbf{b}_2=0$ and $\mathrm{supp}(\mathbf{a}_2)\in \Delta_2$;} \\

                                                                                    T^1(\Delta_2)_{\mathbf{a}_2-\mathbf{b}_2} , & \hbox{$\mathbf{b}_1=0$ and $\mathrm{supp}(\mathbf{a}_1)\in \Delta_1$.}
                                                                                   \end{array}
                                                                                 \right.
\]

 In  particular $\Delta_1 *\Delta_2$
 is rigid  if and only if $\Delta_1$ and $\Delta_2$ are rigid.
\end{Proposition}

\begin{proof}
 \medskip
 If either $\mathrm{supp}(\mathbf{a}_1)\notin \Delta_1$ or $\mathrm{supp}(\mathbf{a}_2)\notin \Delta_2$ then $T^1(\Delta_1\ast \Delta_2)_{\mathbf{a}_1+\mathbf{a}_2-\mathbf{b}_1-\mathbf{b}_2}=0$ by Proposition~\ref{Proposition11}. If  $\mathrm{supp}(\mathbf{a}_1)\in \Delta_1$ and $\mathrm{supp}(\mathbf{a}_2)\in \Delta_2$, we will prove the following equality in several steps.
$$T^1(\Delta_1\ast \Delta_2)_{\mathbf{a}}=(T^1(\Delta_1)[y_1,\ldots,y_m])_{\mathbf{a}}\oplus (T^1(\Delta_2)[x_1,\ldots,x_n])_{\mathbf{a}}.$$
Here, we denote $\mathbf{a}=\mathbf{a}_1+\mathbf{\mathbf{\mathbf{a}}}_2-\mathbf{b}_1-\mathbf{b}_2$, $E_1=\{x_1,\ldots,x_n\}$ and $E_2=\{y_1,\ldots,y_m\}$.  Note that Proposition~\ref{join} follows immediately from this equality.

 (i) First, Observing that links commute with joins, it is enough to prove that
\[
T^1(\Delta_1*\Delta_2)_{-(\mathbf{b}_1+\bb_2)}=(T^1(\Delta_1)[y_1,\ldots,y_m])_{-(\bb_1+\bb_2)}\dirsum (T^1(\Delta_2)[x_1,\ldots,x_n])_{-(\bb_1+\bb_2)},
\]
where $B_i:=\supp\bb_i\subseteq E_i$ for $i=1,2$.

\medskip
(ii) Then we show that $$N_{B_1\union B_2}(\Delta_1*\Delta_2)=[N_{B_1}(\Delta_1)*(\Delta_2\setminus B_2)]\union [(\Delta_1\setminus B_1)* N_{B_2}(\Delta_2)],$$
where $\Delta_i\setminus B_i=\{F\in\Delta_i\:\; F\sect B_i=\emptyset\}$ for $i=1,2$.

\medskip
 Let $F_1\cup F_2\in N_{B_1\union B_2}(\Delta_1*\Delta_2)$ with $F_i\in \Delta_i$ for $i=1,2$. Then $F_i\in \Delta_i\setminus B_i$ for $i=1,2$ and $(F_1\cup F_2)\cup(B_1\cup B_2)\notin \Delta_1*\Delta_2$ by definition. It follows that for at least one $i\in \{1,2\}$, $F_i\cup B_i\notin \Delta_i$, namely, $F_i\in N_{B_i}(\Delta_i)$. This actually proves the containment $\subseteq$. The other containment  can be proved similarly.

\medskip
(iii)  Suppose that $B_i\neq \emptyset$ for $i=1,2$.  We will show that both sides of the identity in step (i) vanishes.

Let $F_1\cup F_2\in N_{B_1\union B_2}(\Delta_1*\Delta_2)$ with $F_i\in \Delta_i$ for $i=1,2$. Then $F_i\in N_{B_i}(\Delta_i)$,  for at least one $i\in \{1,2\}$ by the proof of (ii), say $i=1$. It follows that $(F_1\cup F_2)\cup B_1\notin \Delta_1*\Delta_2$. Since $B_1\subsetneq B_1\cup B_2$, we have $F_1\cup F_2\in \widetilde{N}_{B_1\union B_2}(\Delta_1*\Delta_2)$. Hence $N_{B_1\union B_2}(\Delta_1*\Delta_2)=\widetilde{N}_{B_1\union B_2}(\Delta_1*\Delta_2)$. This implies that the left side of the identity in step (i) vanishes.

To see that the right side also vanishes, we note that a multi-homogeneous element in $T^1(\Delta_1)[y_1,\ldots,y_m]$ has the form $ty_1^{a_1}\cdots y_m^{a_m}$, where $t$ is a multi-homogeneous element in $T^1(\Delta_1)$ and $a_i\geq 0$ for $i=1,\ldots,m$. But $y_1^{a_1}\cdots y_m^{a_m}$ has the
multi-homogeneous degree for which each coordinate is greater or equal than zero. Hence $(T^1(\Delta_1)[y_1,\ldots,y_m])_{-(\bb_1+\bb_2)}=0$. Similarly  $(T^1(\Delta_2)[x_1,\ldots,x_n])_{-(\bb_1+\bb_2)}=0$.

\medskip
(iv) If one of the $B_i$ is the empty set, say $B_2 =\emptyset$, then  (ii) implies that
\[
N_{B_1}(\Delta_1*\Delta_2)=N_{B_1}(\Delta_1)*\Delta_2 \text{ and }  \widetilde{N}_{B_1}(\Delta_1*\Delta_2)=\widetilde{N}_{B_1}(\Delta_1)*\Delta_2.
\]
 This   yields $T^1(\Delta_1*\Delta_2)_{-\bb_1}= T^1(\Delta_1)_{-\bb_1}=(T^1(\Delta_1)[y_1,\ldots,y_m])_{-\bb_1}$, as desired.
 \end{proof}

Proposition~\ref{join} can be rewritten as the following nice formula, which is suggested by William Bitsch, whose email is  william.bitsch@fu-berlin.de.

\begin{Corollary} For all $\mathbf{z}\in\ZZ^{E_1\cup E_2}$, we have
\[
T^1(\Delta_1\ast \Delta_2)_{\mathbf{z}}=(T^1(\Delta_1)\otimes k[\Delta_2])_{\mathbf{z}}\oplus (T^1(\Delta_1)\otimes k[\Delta_2])_{\mathbf{z}}\qquad .
\]
\end{Corollary}

\begin{proof} First we write $\mathbf{z}$ as $\mathbf{a}_1+\mathbf{a}_2-\mathbf{b}_1-\mathbf{b}_2$ with $\mathrm{supp}(\mathbf{a}_i)\cap \mathrm{supp}(\mathbf{b}_i)=\emptyset$,  $\mathrm{supp}(\mathbf{\mathbf{a}}_i)\cup \mathrm{supp} (\mathbf{b}_i)\subseteq E_i$,  $\mathbf{a}_i\in \mathbb{N}^{E_1\cup E_2}$ and $\mathbf{b}_i\in \{0,1\}^{E_1\cup E_2}$ for $i=1,2$.

\vspace{3mm}

Note that as a multigraded linear $k$-space,  $$k[\Delta_2]\cong \bigoplus_{\mathbf{a}\in \mathbb{N}^{E_2}, \mathrm{supp}(\mathbf{a})\in \Delta_2} ky^{\mathbf{a}}.$$

Hence $$(T^1(\Delta_1)\otimes k[\Delta_2])_{\mathbf{a}_1+\mathbf{a}_2-\mathbf{\mathbf{b}}_1-\mathbf{b}_2}=\bigoplus_{\mathbf{a}\in \mathbb{N}^{E_2}, \mathrm{supp}(\mathbf{a})\in \Delta_2}T^1(\Delta_1)_{\mathbf{a}_1+\mathbf{a}_2-\mathbf{a}-\mathbf{b}_1-\mathbf{b}_2}$$

For $\mathbf{a}\in \mathbb{N}^{E_2}$, it is easy to see that $T^1(\Delta_1)_{\mathbf{a}_1+\mathbf{a}_2-\mathbf{a}-\mathbf{b}_1-\mathbf{b}_2}\neq 0$ if and only $T^1(\Delta_1)_{\mathbf{a}_1-\mathbf{b}_1}\neq 0$  and $\mathbf{a}= \mathbf{a}_2, \mathbf{b}_2=0.$ This implies  $$(T^1(\Delta_1)\otimes k[\Delta_2])_{\mathbf{a}_1+\mathbf{a}_2-\mathbf{b}_1-\mathbf{b}_2}\cong \left\{
                                                                                   \begin{array}{ll}
                                                                                     T^1(\Delta_1)_{\mathbf{a}_1-\mathbf{b}_1}, & \hbox{ $\mathbf{b}_2= 0$ and $\mathrm{supp}(\mathbf{a}_2)\in \Delta_2\  $;} \\
                                                                                     0, & \hbox{otherwise}
                                                                                    \end{array}
                                                                                 \right.
$$  From this the fomula follows in view of  Proposition~\ref{join}.
 \end{proof}

\bigskip
\noindent
{\it Disjoint unions.} Next we consider the simplicial complex $\Delta_1\union \Delta_2$ which is the disjoint union of $\Delta_1$ and $\Delta_2$. The vertex set of $\Delta_1\union \Delta_2$ is $V(\Delta_1)\union V(\Delta_2)$ and $F \in \Delta_1\union \Delta_2$ if and only if $F\in \Delta_1$ or $F\in \Delta_2$.

\begin{Theorem}
\label{union}
Let $\Delta_1\neq\{\emptyset\}$ and $\Delta_2\neq \{\emptyset\}$ be simplicial complexes with disjoint vertex sets.  Assume that for $i=1,2$, $V(\Delta_i)=[\Delta_i]$, that is, $\{j\}\in \Delta_i$ for all $j\in V(\Delta_i)$.
\begin{enumerate}
\item[(a)] The following conditions are equivalent:
\begin{enumerate}
\item[(1)] $\Delta_1\union \Delta_2$ is inseparable;
\item[(2)] $\Delta_1$ and $\Delta_2$ are simplices.
\end{enumerate}
\item[(b)] The following conditions are equivalent:
\begin{enumerate}
\item[(1)] $\Delta_1 \union \Delta_2$ is rigid;
\item[(2)] $\Delta_1\union \Delta_2$ is $\emptyset$-rigid;
\item[(3)] $\Delta_1$ and $\Delta_2$ are simplices with  $\dim \Delta_1 +\dim \Delta_2>0$.
\end{enumerate}
\end{enumerate}
\end{Theorem}

\begin{proof}
Let $\Delta=\Delta_1\cup \Delta_2$.

(a) (1)\implies(2):  Since $\Delta$ is inseparable,  Theorem~\ref{insep} implies that the graph $G_{\{i\}}(\Delta)$ is connected for all $i\in V(\Delta_1)\cup V(\Delta_2)$. Let $i\in V(\Delta_1)$.  By assumption we have $\{i\}\in \Delta$. It follows that  $N_{\{i\}}(\Delta)= N_{\{i\}}(\Delta_1)\cup \Delta_2\setminus \{\emptyset\}$. Since  $N_{\{i\}}(\Delta_1)$ and $\Delta_2\setminus \{\emptyset\}$ belong to different connected components  of  $G_{\{i\}}(\Delta)$ and since $G_{\{i\}}(\Delta)$  is connected, it follows that   either $N_{\{i\}}(\Delta_1)=\emptyset$ or $\Delta_2\setminus \{\emptyset\}=\emptyset$. The second case is ruled out by assumption. Hence, $N_{\{i\}}(\Delta_1)=\emptyset$. This implies that $i\in F$ for all $F\in \MF(\Delta_1)$. Since $i$ is an arbitrary element in $V(\Delta_1)$ we see that $\MF(\Delta_1)=\{V(\Delta_1)\}$. Starting with $i\in V(\Delta_2)$, the same argument proves that $\Delta_2$ is also a simplex.


(2)\implies (1): By Theorem~\ref{insep}, it is enough to show that $G_{\{i\}}(\Delta)$ is connected for all $i\in V(\Delta_1)\cup V(\Delta_2)$. Let $i\in V(\Delta_1)$.  As mentioned above we have $N_{\{i\}}(\Delta)= N_{\{i\}}(\Delta_1)\cup \Delta_2\setminus \{\emptyset\}$. Since   $\Delta_1$ is a simplex and $\{i\}\in \Delta_1$ it follows that $N_{\{i\}}(\Delta_1)=\emptyset$. Thus  $N_{\{i\}}(\Delta)= \Delta_2\setminus \{\emptyset\}$. Therefore $G_{\{i\}}(\Delta)$ is connected because $\Delta_2$ is a simplex. A similar argument shows that $G_{\{i\}}(\Delta)$ is also  connected for all $i\in V(\Delta_2)$.

(b) (1)\implies (2) is obvious.

(2)\implies (3):  Since $\Delta$ is $\emptyset$-rigid we have $T^1(\Delta)_{-\bb}=0$ for all $\bb\in \{0,1\}^{V(\Delta_1)\cup V(\Delta_2)}$. In particular, $T^1(\Delta)_{-\eb_i}=0$ for all $i\in V(\Delta_1)\cup V(\Delta_2)$. So by Theorem~\ref{insep}, $\Delta$ is inseparable. Thus using  part (a) we have that $\Delta_1$ and $\Delta_2$ are simplices. Suppose that  $\dim \Delta_1 +\dim \Delta_2=0$. Then $\dim \Delta_1=\dim \Delta_2=0$, and hence $I_\Delta$ is of the form $(xy)$. It follows that $T^1(\Delta)_{-(1,1)}\neq 0$, a contradiction.

(3)\implies (1): The assumptions imply that $I_\Delta$ is of the form $(x_iy_j\: i=1,\ldots,n, \; j=1,\ldots,m)$ with $n\geq 2$ or $m\geq 2$. Corollary~\ref{product} implies that $\Delta$ is rigid.
\end{proof}

\begin{Corollary}
\label{flower}
Let $\Delta$ be  simplicial complex. Then $\Delta$ is separable if $\Delta$ has more than two connected components.
\end{Corollary}

\bigskip
\noindent
{\it Circs.} For $i=1,2$, let $\Delta_i$ be a  simplicial complex on the vertex set $V_i$ and assume that $V_1\sect V_2=\emptyset$.  Then  the {\em circ}  of $\Delta_1$ and $\Delta_2$ is the simplicial complex $\Delta_1\circ\Delta_2$  with vertex set $V_1\cup V_2$ whose faces  are those subsets $F$ of $V_1\cup V_2$ for which either   $F\cap V_1$ is a face of $\Delta_1$ or  $F\cap V_2$ is a face of $\Delta_2$.

It is worthwhile to note that if $I_{\Delta_1}\subseteq S_1=K[x_1,\ldots,x_n]$ and $I_{\Delta_2}\subseteq S_2=K[y_1,\ldots,y_m]$ then
$I_{\Delta_1\circ \Delta_2}=I_{\Delta_1}I_{\Delta_2}S$, where $S=K[x_1,\ldots,x_n,y_1\,\ldots,y_m]$.

\medskip
In the following we set  $M_B(\Delta)=\{F\subseteq V(\Delta)\:\; F\notin \Delta \mbox{\ and\ } F\cap B=\emptyset\}$ for any $B\subseteq V(\Delta)$.
For later use we list a few obvious facts in the next lemmata.
\medskip

\begin{Lemma}
 \label{N}

 \vspace{2mm}

Let $B=B_1\cup B_2$ with $\emptyset \neq B_i\subseteq V_i$ for $i=1,2$. Then
\begin{enumerate}
\item[(a)]
$ N_B(\Delta_1\circ\Delta_2)=N_{B_1}(\Delta_1)*N_{B_2}(\Delta_2)\union N_{B_1}(\Delta_1)*M_{B_2}(\Delta_2)\newline
\mbox{\qquad \qquad \qquad \quad} \union M_{B_1}(\Delta_1)*N_{B_2}(\Delta_2).$
\vspace{2mm}

\item[(b)]  $N_{B_1}(\Delta_1\circ\Delta_2)=N_{B_1}(\Delta_1)*M_{\emptyset}(\Delta_2)$ and $N_{B_2}(\Delta_1\circ\Delta_2)=M_{\emptyset}(\Delta_1)*N_{B_2}(\Delta_2)$.

\end{enumerate}
 \end{Lemma}

\vspace{2mm}

\begin{Lemma}
\label{tildeN}
Let  $B=B_1\cup B_2$ with $\emptyset \neq B_i\subseteq V_i$ for $i=1,2$. Then
\begin{itemize}

\vspace{2mm}

\item[(a)]
$\widetilde{N}_B(\Delta_1\circ\Delta_2)=N_{B_1}(\Delta_1)*\widetilde{N}_{B_2}(\Delta_2)\cup \widetilde{N}_{B_1}(\Delta_1)*N_{B_2}(\Delta_2)\newline
 \mbox{\qquad \qquad \qquad\quad} \union  N_{B_1}(\Delta_1)*M_{B_2}(\Delta_2)\union M_{B_1}(\Delta_1)*N_{B_2}(\Delta_2).$
\vspace{2mm}

\item[(b)]  $\widetilde{N}_{B_1}(\Delta_1\circ\Delta_2)=\widetilde{N}_{B_1}(\Delta_1)*M_{\emptyset}(\Delta_2)$ and $\widetilde{N}_{B_2}(\Delta_1\circ\Delta_2)=M_{\emptyset}(\Delta_1)*\widetilde{N}_{B_2}(\Delta_2)$.
\end{itemize}
\end{Lemma}

We use these lemmata to prove

\begin{Proposition}
\label{carneval}
Let $\Delta_1$ and $\Delta_2$ be two simplicial complexes on the vertex sets $V_1$ and $V_2$, respectively, and let $\bb_i\in \{0,1\}^{V_i}$,  $B_i=\supp \bb_i$ for $i=1,2$. Assume that $I_{\Delta_i}\neq 0$ for $i=1,2$.
\begin{enumerate}
\item[(a)] Suppose that $\bb_1, \bb_2\neq 0$. Then
\begin{eqnarray*}
T^1(\Delta_1\circ\Delta_2)_{-(\bb_1+ \bb_2)}\iso
\left\{
\begin{array}{ll}
K, &\mbox{if $\widetilde{N}_{B_1\cup B_2}(\Delta_1\circ\Delta_2)=\emptyset$ and $N_{B_1\cup B_2}(\Delta_1\circ\Delta_2)\neq\emptyset$, }\\
0, &\mbox{otherwise}.
\end{array}
\right.
\end{eqnarray*}

\item[(b)] $\dim_K T^1(\Delta_1\circ\Delta_2)_{-\bb_i}=\dim_K T^1(\Delta_i)_{-\bb_i }$ for $i=1,2$.
\end{enumerate}
\end{Proposition}

\begin{proof} First we observe that for each $i=1,2$,  $I_{\Delta_i}\neq 0$ if and only if $V_i\notin\Delta_i$, which is equivalent to  $M_{\phi}(\Delta_i)\neq \emptyset$.

 (a) Assume  that $\widetilde{N}_{B_1\cup B_2}(\Delta_1\circ\Delta_2)=\emptyset$ and $N_{B_1\cup B_2}(\Delta_1\circ\Delta_2)\neq\emptyset$. Lemma~\ref{N}(a)  together with Lemma~\ref{tildeN}(a) imply that $N_{B_1}(\Delta_1)*N_{B_2}(\Delta_2)\neq \emptyset$. Therefore $N_{B_i}(\Delta_i)\neq \emptyset$ for $i=1,2$.
Since $\widetilde{N}_{B_1\cup B_2}(\Delta_1\circ\Delta_2)=\emptyset$  it follows from Lemma~\ref{tildeN}(a) that $N_{B_1}(\Delta_1)*\widetilde{N}_{B_2}(\Delta_2)=\emptyset$ and  $\widetilde{N}_{B_1}(\Delta_1)*N_{B_2}(\Delta_2)=\emptyset$, and so $\widetilde{N}_{B_i}(\Delta_i)=\emptyset$ for $i=1,2$. The same argument shows that  $M_{B_i}(\Delta_i)=\emptyset$ for $i=1,2$.

 Since $M_{B_i}(\Delta_i)=\emptyset $ for $i=1,2$, we have $V_i\setminus B_i\in \Delta_i$ and  hence $V_i\setminus B_i\in N_{B_i}(\Delta_i)$ for $i=1,2$.  It follows that  any $\lambda\in T^1(\Delta_1\circ\Delta_2)_{-(\bb_1+ \bb_2)}$ is constant on $N_{B_1\cup B_2}(\Delta_1\circ\Delta_2)$, because  for any $F_1\cup F_2\in N_{B_1\cup B_2}(\Delta_1\circ \Delta_2) $ with $F_i\subseteq V_i$ for $i=1,2$ we have  $F_1\cup F_2\subseteq (V_1\setminus B_1) \cup (V_2\setminus B_2).$
  Hence $T^1(\Delta_1\circ\Delta_2)_{-(\bb_1+ \bb_2)}\cong K$.

 If $N_{B_1\cup B_2}(\Delta_1\circ\Delta_2)=\emptyset,$ then  $T^1(\Delta_1\circ\Delta_2)_{-(\bb_1+ \bb_2)}=0$. Assume now that  $\widetilde{N}_{B_1\cup B_2}(\Delta_1\circ\Delta_2)\neq \emptyset$.  Let $\lambda\in \Lambda^1(\Delta_1\circ\Delta_2)_{-(\bb_1+ \bb_2)}$ (see its definition in  Corollary~\ref{useful})  and $F_1\cup F_2\in N_{B_1\cup B_2}(\Delta_1\circ \Delta_2)\setminus  \widetilde{N}_{B_1\cup B_2}(\Delta_1\circ\Delta_2)$ with $F_i\subseteq V_i,i=1,2$. It follows from Lemmata~\ref{N}(a) and \ref{tildeN}(a) that $F_i\in N_{B_i}(\Delta_i)\setminus \widetilde{N}_{B_i}(\Delta_i)$ for $i=1,2$. By Lemma~\ref{tildeN}(a), at least one of $\widetilde{N}_{B_i}(\Delta_i), i=1,2$ and
  $M_{B_i}(\Delta_i), i=1,2$ is nonempty.

 First suppose that $M_{B_1}(\Delta_1)\neq \emptyset$. Then we take  $G_1\in M_{B_1}(\Delta_1)$. Note that  $G_1\cup F_2\in \widetilde{N}_{B_1\cup B_2}(\Delta_1\circ\Delta_2)$ and $(G_1\cup F_2)\cup (F_1\cup F_2)\in \Delta_1\circ \Delta_2$. Therefore $\lambda(F_1\cup F_2)=0$. Similarly, we can conclude that  $\lambda(F_1\cup F_2)=0$ if $M_{B_2}(\Delta_2)\neq \emptyset$ or $\widetilde{N}_{B_i}(\Delta_i)\neq \emptyset$ for $i=1$ or $i=2$. So we have proved that  $\lambda=0$ in any case.  Thus $T^1(\Delta_1\circ\Delta_2)_{-(\bb_1+ \bb_2)}=0$.

(b) We prove the statement for $i=1$. The same argument holds for $i=2$.
Since $M_{\emptyset}(\Delta_2)\neq \emptyset$, we have $N_{B_1\cup B_2}(\Delta_1\circ\Delta_2)=\emptyset$ if and only if $N_{B_1}(\Delta_1)=\emptyset$ by Lemma~\ref{N}(b).
Hence we only need to consider the case when $N_{B_1\cup B_2}(\Delta_1\circ\Delta_2)\neq \emptyset$.

 We define a homomorphism $\varphi$ from $\Lambda^1(\Delta_1)_{-\bb_1}$ to   $\Lambda^1(\Delta_1\circ\Delta_2)_{-\bb_1}$  as follows: for $\lambda_1\in \Lambda^1(\Delta_1)_{-\bb_1}$, $\varphi(\lambda_1)$ is given by $\varphi(\lambda_1)(F_1\cup G_1)=\lambda_1(F_1)$ for any $F_1\in N_{B_1}(\Delta_1)$ and $G_1\in M_{\emptyset}(\Delta_2).$  One can check that $\varphi(\lambda_1)$ belongs to $\Lambda^1(\Delta_1\circ\Delta_2)_{-\bb_1}$ and that $\varphi$ is injective. To see that  $\varphi$ is surjective, one only need to notice that for any $\lambda\in \Lambda^1(\Delta_1\circ\Delta_2)_{-\bb_1}$, $\lambda(F_1\cup G_1)=\lambda(F_1\cup G_2)$ for any $F_1\in N_{B_1}(\Delta_1)$ and $G_1,G_2\in M_{\emptyset}(\Delta_2).$ Hence $\varphi$ is an isomorphism of $K$-vector spaces. This implies that  $\dim_K T^1(\Delta_1\circ\Delta_2)_{-\bb_1}= \dim_K T^1(\Delta_1)_{-\bb_1}$, using Corollary~\ref{useful}.
\end{proof}

For simplicity we say that a simplicial complex $\Delta$ on the vertex set $[n]$ is {\em special} if either $\Delta$ has a unique facet, which is a maximal proper subset of $[n]$,  or $\Delta$ has exactly two facets, one of which contains a unique element say $i$ and the other one of which is $[n]\setminus \{i\}$. We see that $\Delta$ is special if and only if  $I_{\Delta}$ is of the form $zP$,  where $z$ is a variable in $S$ with $z\not \in P$, and $P$ is either a monomial prime ideal of $S$ or $P=S$.
\medskip

\begin{Theorem}
\label{circ}
{\em (a)} Suppose that $\Delta_1\circ\Delta_2$ is rigid, and that $I_{\Delta_i}\neq 0$ for $i=1,2$. Then $\Delta_1$ and $\Delta_2$ are rigid.

{\em (b)} Conversely, suppose that $\Delta_1$ and $\Delta_2$ are  rigid simplicial complexes with disjoint  vertex sets $V_1$ and $V_2$, respectively. Then $\Delta_1\circ \Delta_2$ is rigid if and only if for either $j=1$ or $j=2$, $\link_{\Delta_j}F$ is not special for all  $F\in \Delta_j$. In algebraic terms,  for $j=1$ or $j=2$, none of the monomial localizations of $I_{\Delta_j}$  is of the form $zP$, where $z$ is a variable in $S_j$ with $z\not \in P$, and $P\subseteq S_j$ is either a monomial prime ideal or $P=S_j$.
\end{Theorem}

\begin{proof}  Let $S_1=K[x_1,\ldots,x_n]$, $S_2=K[y_1,\ldots,y_m]$ and $S=K[x_1,\ldots,x_n,y_1,\ldots,y_n]$.
We may assume that $I_{\Delta_1}\subseteq S_1$ and $I_{\Delta_2}\subseteq S_2$.

(a) Let $P_1=(x_1,\ldots,x_n)$ and $P_2=(y_1,\ldots,y_m)$.  Since $$I_{\Delta_1\circ \Delta_2}(P_i)=(I_{\Delta_1} I_{\Delta_2}S)(P_i)=I_{\Delta_i}$$ for $i=1,2$, we see  that
\[
T^1(K[\Delta_1\circ \Delta_2](P_i))\iso T^1(K[\Delta_i]).
\]
Thus the assertion follows from Lemma~\ref{local}(b).

(b) First suppose that $\Delta_1\circ\Delta_2$ is rigid, and suppose that there exist $F_1\in\Delta_1$ and $F_2\in \Delta_2$ such that the  Stanley--Reisner ideal $I_{\link_{\Delta_j}F_j}$ is  of the form $z_jP_j$ where  $z_j$ is a variable in $S_j$ with $z_j\not \in P_j$,
$P_j\subseteq S_j$ is either a monomial prime ideal or $P_j=S_j$. Since $\link_{\Delta_1\circ \Delta_2}F_1\cup F_2=\link_{\Delta_1}F_1\circ \link_{\Delta_2}F_2$, it follows that
$$(I_{\link_{\Delta_1\circ \Delta_2}F_1\cup F_2})S=I_{\link_{\Delta_1}F_1}I_{\link_{\Delta_2}F_2}S=(z_1z_2)P_1P_2S$$
is not rigid because the ideal $(z_1z_2)$ is not rigid, a contradiction (see Proposition~\ref{Proposition11}(b) or Lemma~\ref{local}(c)).

 Suppose that $\Delta_1\circ \Delta_2$ is not rigid. Then there exist $\ab,\bb\in \{0,1\}^{V_1\cup V_2}$ with $\supp \ab\sect \supp \bb= \emptyset$ such that  $T^1(\Delta_1\circ \Delta_2)_{\ab-\bb}\neq 0$. Let $\ab=\ab_1+ \ab_2$ and $\bb=\bb_1+ \bb_2$ with $\ab_1, \bb_1\in \{0,1\}^{V_1}$ and $\ab_2, \bb_2\in\{0,1\}^{V_2}$.  We set $B=\supp \bb$, $B_i=\supp \bb_i$ for $i=1,2$, and denote by $\Gamma_i$ the simplicial complex $\link_{\Delta_i}(\supp \ab_i)$ for $i=1,2$. Since $\link_{\Delta_1\circ\Delta_2}(\supp \ab)=\Gamma_1\circ \Gamma_2$,  we have
$T^1(\Gamma_1\circ \Gamma_2)_{-\bb} \neq 0$ by using Proposition~\ref{Proposition11}(b). Note that $\Gamma_i$ is rigid for $i=1,2$ by Lemma~\ref{local}(c).

Since $T^1(\Gamma_1\circ \Gamma_2)_{-\bb} \neq 0$ we have $(I_{\Gamma_1}I_{\Gamma_2})=I_{\Gamma_1\circ\Gamma_2}\neq 0$. This implies that $I_{\Gamma_i}\neq 0$ for $i=1,2$.
Suppose $\bb_1=0$. Since $T^1(\Gamma_1\circ \Gamma_2)_{-\bb_2}=T^1(\Gamma_2)_{-\bb_2}$ (see Proposition~\ref{carneval}(b)), we have $T^1(\Gamma_1\circ \Gamma_2)_{-\bb_2} = 0$, because $\Gamma_2$ is rigid. This is a contradiction. Therefore $\bb_1\neq 0$, and similarly $\bb_2\neq 0$.
Now Proposition~\ref{carneval}(a) implies that ${N}_{B_1\cup B_2}(\Gamma_1\circ \Gamma_2)\neq\emptyset$ and $\widetilde{N}_{B_1\cup B_2}(\Gamma_1\circ \Gamma_2)=\emptyset$.  By a similar argument used in Proposition~\ref{carneval}, we have $N_{B_i}(\Gamma_i)\neq \emptyset$  and $M_{B_i}(\Gamma_i)=\widetilde{N}_{B_i}(\Gamma_i)= \emptyset$ for $i=1,2$. Therefore, for $i=1,2$  the map $\lambda_i : N_{B_i}(\Gamma_i)\to K$, which maps each element of $N_{B_i}(\Gamma_i)$ to $1$, defines a nonzero element in $\Lambda(\Gamma_i)_{-\bb_i}$. Since  $\Gamma_i$ is rigid it follows from Corollary~\ref{useful}(a) that  $|B_i|=1$  for $i=1,2$. Without loss of generality we may assume that  $B_{1}=\{x_1\}$ and $B_2=\{y_1\}$. Then $V(\Gamma_1)\setminus \{x_1\}\in \Gamma_1$ and  $V(\Gamma_2)\setminus \{y_1\}\in \Gamma_2$ since $M_{B_i}(\Gamma_i)= \emptyset$ for $i=1,2$. Now we consider the following cases:

Case 1: $|V(\Gamma_1)|=1$. Since $I_{\Gamma_1}$ is not zero  we have $I_{\Gamma_1}=(x_1)$ and so  $\Gamma_1$ is special.

Case 2: $|V(\Gamma_1)|=2$. Assume $V(\Gamma_1)=\{x_1,x_2\}$. Then $\{x_2\}\in \Gamma_1$. Since $I_{\Gamma_1}\neq 0$, we have $\{x_1,x_2\}\notin \Gamma_1$ and since $I_{\Gamma_1}$ is rigid we have $\Gamma_1\neq \langle\{x_1\},\{x_2\}\rangle$. Hence $\Gamma_1=\langle\{x_2\}\rangle$ and $I_{\Gamma_1}=(x_1)$.  This implies that $\Gamma_1$ is special again.

Case 3:   $|V(\Gamma_1)|\geq 3$. Since $V(\Gamma_1)\setminus \{x_1\}\in \Gamma_1$, we have either $\Gamma_1=\langle V(\Gamma_1)\setminus \{x_1\} \rangle$ or $\Gamma_1=\langle V(\Gamma_1)\setminus \{x_1\}, F_1,\ldots, F_k \rangle$, where $k\geq 1$ and $x_1\in F_j$ for $j=1,\ldots, k$. In the first case we have $I_{\Gamma_1}=(x_1)$ and we are  done. In the second  one, we set $G=F_1\setminus \{x_1\}$. Then $\link_{\Gamma_1}G=\langle V(\Gamma_1)\setminus F_1,\{x_1\}\rangle$ and $V(\link_{\Gamma_1}G)=(V(\Gamma_1)\setminus F_1)\cup \{x_1\}$. Write $V(\Gamma_1)\setminus F_1=\{x_2,\ldots,x_s\}$. Since $\link_{\Gamma_1}G$ is rigid, we have $s\geq 3$ and so $I_{\link_{\Gamma_1}G}=(x_1)(x_2,\ldots, x_s)=x_1P$. This also implies that $\Gamma_1$ is special.

The similar argument is applied to $\Gamma_2$.
\end{proof}

\begin{Corollary}
\label{product}
{\em(a)} Let $I$ be a squarefree monomial ideal and let $P$ be a non-principal monomial prime ideal in a disjoint set of variables. Then $IP$ is rigid if and only if $I$ is rigid.

{\em(b)} Let $P_1,\ldots, P_t$ be monomial prime ideals generated by pairwise disjoint sets of variables.  Then $\prod_{i=1}^tP_i$ is rigid if and only if at most one $P_i$ is a principal ideal.
\end{Corollary}
\begin{proof}
(a) Let $\Delta_1$ and $\Delta_2$ be two simplicial complexes with $I_{\Delta_1}=I$ and $I_{\Delta_2}=P$. Suppose $IP$ is rigid.  It follows from Theorem~\ref{circ}(a) that $I$ is rigid.
Conversely, suppose that $I$ is rigid. Note that the links of $\Delta_2$ correspond to monomial localizations of $P$ by Lemma~\ref{local}(c).
Since monomial localization of $P$ with respect to any monomial prime ideal is never of the form $(z)$ or $zQ$  with $Q$  a monomial prime ideal, Theorem~\ref{circ}(b) yields the desired conclusion.

(b) follows immediately from (a).
\end{proof}

\bigskip
\noindent
{\it Letterplace ideals.} We conclude this section with applications to letterplace ideals.
In \cite{FGH}, letterplace and co-letterplace ideals are introduced and it is shown that these are all inseparable monomial ideals. In this section we consider rigidity of this class of ideals.

More generally, let $\MP$ and $\MQ$ be two partially ordered sets. A map $\varphi: \MP \to \MQ$
is called {\it isotone} or {\it order preserving}, if $p\leq p^\prime$
implies $\varphi(p) \leq \varphi(p^\prime)$. The set of isotone maps is denoted by
$\Hom(\MP,\MQ)$. Note that $\Hom(\MP,\MQ)$ is  again a partially ordered set
with $\phi \leq \psi$ if $\phi(p) \leq \psi(p)$ for all $p \in \MP$.

We fix a field $K$ and consider the polynomial ring  $S$ over $K$ in the variables $x_{p,q}$ with $p\in \MP$ and $q\in \MQ$. Attached to $\MP$ and $\MQ$ we define the monomial ideal $L(\MP,\MQ)\subseteq S$ generated by the monomials
\[
u_\varphi=\prod_{p\in \MP}x_{p,\varphi(p)}, \quad \varphi\in \Hom(\MP,\MQ).
\]

\begin{Theorem}
\label{dan}
Let $\MP$ and $\MQ$ be  finite posets. The following conditions  are equivalent:
\begin{enumerate}
\item[(a)]  $L(\MP, \MQ)$ is  rigid.
\item[(b)] No two distinct elements of $\MP$ are comparable.
\end{enumerate}
\end{Theorem}

\begin{proof}

 (a)\implies (b): Assume there exist $a,b\in \MP$ with $a<b$. We consider the monomial prime ideal
 $P=(x_{p,q}:\; p\in\{a,b\},q\in \MQ)$, and  claim that $L(\MP,\MQ)(P)=L(\{a,b\},\MQ)$, where $\{a,b\}$ is the  poset with $a<b$.  In fact, for any minimal generator $u\in L(\MP,\MQ)(P)$, there exists  $\varphi\in \Hom(\MP,\MQ)$ such that $u$ is obtained from $u_{\varphi}$ by setting $x_{p,\varphi(p)}=1$ if $p\notin \{a,b\}$, that is, $u=x_{a,\varphi(a)}x_{b,\varphi(b)}$. This proves $L(\MP,\MQ)(P)\subseteq L(\{a,b\},\MQ)$.

Conversely, let $u=x_{a,c}x_{b,d}\in L(\{a,b\},\MQ)$, where $c,d\in\MQ$ and $c\leq d$. Let $n=|\MP|$. Since any finite partial order can be extended to a total order,  there exists  an isotone bijective map from $\MP$ to $[n]$, which we denote by $f$. We now  define a map $\varphi:\MP\to \MQ$ as follows:
\[
\varphi(p)=
\left\{
\begin{array}{ll}
c, &\mbox{if $f(p)<f(b)$},\\
d, &\mbox{otherwise}.
\end{array}
\right.
\]
Then   $\varphi\in\Hom(\MP,\MQ)$ and $\varphi(a)=c,\varphi(b)=d$, and  hence $u\in L(\MP,\MQ)(P)$. Thus our claim follows.

It follows from Corollary~\ref{CM} and its proof that $S(P)/L(\{a,b\},\MQ)$ is not rigid. Therefore, Lemma~\ref{local}(b) implies that   $S/L(\MP,\MQ)$ is not rigid.

(b)\implies(a): Let $\MP=\{p_1,\ldots,p_m\}$ with $p_i$ and $p_j$ incomparable for all $i\neq j$ and let $\MQ=\{q_1,\ldots,q_n\}$. Then
\[
L(\MP,\MQ)=\prod_{i=1}^m(x_{p_i,q_1},x_{p_i,q_2},\ldots,x_{p_i,q_n}).
\]
Thus the assertion follows from Corollary~\ref{product}(b).
\end{proof}

For an integer $n\in \NN$ we denote by $[n]$ the totally ordered set $\{1<2<\cdots <n\}$. The ideal $L([n],\MP)$ is called the {\it $n$th  letterplace ideal}, while $L(\MP,[n])$ is called the {\it $n$th co-letterplace ideal}. They are Alexander dual to each other  if we identify $x_{i,p}$ with $x_{p,i}$ for any $i\in [n]$ and $p\in \MP$. In particular, the facets of the simplicial complex associated with $L([n],\MP)$ are in bijection with the generators of $L(\MP,[n])$, and vice versa.

\begin{Corollary}
\label{letter}
Let $\MP$  a finite poset.
\begin{enumerate}
\item[(a)] $L([n],\MP)$ is rigid if and only if $n=1$.
\item[(b)] $L(\MP,[n])$ is rigid if and only if no two distinct elements of $\MP$ are comparable.
\end{enumerate}
\end{Corollary}

\section{Rigidity of graphs}
\label{graphs}

In this section we apply the results of Section~\ref{rigidity} to study the rigidity of edge ideals of a graph.

\bigskip
\noindent
{\it Inseparable graphs.} Let $G$ be a finite simple graph with vertex set $[n]$. The edge set of $G$ will be denoted by $E(G)$. Let $K$ be a field and $S=K[x_1,\ldots, x_n]$ the polynomial ring over $K$ in $n$ indeterminates $x_1,\ldots,x_n$.
The {\em edge ideal} $I(G)\subseteq S$
 of $G$ is defined to be the ideal generated by all products $x_ix_j$ with
$\{i,j\}\in E(G)$. Let $\Delta(G)$ be the simplicial complex with $I(G)=I_{\Delta(G)}$ Then
$$
\Delta(G)=\{F\subseteq [n]\:\; F \mbox{ does not contain any edges of } G\},
$$
The simplicial complex $\Delta(G)$ is called the {\em independence complex} of $G$.  The faces of $\Delta(G)$ are called the {\em independent sets} of $G$.

We call $G$ {\em inseparable} if $I(G)$ is inseparable. Let $i\in [n]$. Then $N(i)=\{j\:\; \{j,i\}\in E(G)\}$ is called the {\em neighborhood} of $i$. We denote by $G^{(i)}$ the complementary graph of the restriction $G_{N(i)}$ of $G$ to $N(i)$. In other words, $V(G^{(i)})=N(i)$ and $E(G^{(i)})=\{\{j,k\}\:\; j\neq k, j,k\in N(i) \text{ and } \{j,k\}\not\in E(G)\}$.  Note that $G^{(i)}$ is disconnected if and only if  $N(i)=A\cup B$, where $A,B\neq \emptyset$, $A\sect B=\emptyset$ and all  vertices of $A$ are adjacent to those of $B$.

\begin{Theorem}
\label{combinatorially}
 The following conditions are equivalent:
\begin{enumerate}
\item[(a)] The  graph $G$ is inseparable;

\item[(b)] $G^{(i)}$ is connected for all $i$;

\item[(c)]  $T^1(S/I(G))_{-\eb_i}=0$ for all $i$.
\end{enumerate}
\end{Theorem}

\begin{proof} We set $\Delta=\Delta(G)$. By Theorem~\ref{insep}, it suffices to prove that $G_{\{i\}}(\Delta)$ is connected if and only if $G^{(i)}$ is connected for each $i$. First we note that  $V(G^{(i)})\subseteq [n]$ and $$V(G_{\{i\}}(\Delta))=\{F\subseteq [n]\:\; F \mbox{ is an independent set of } G \mbox{ and } F\cap V(G^{(i)})\neq \emptyset  \}.$$  Assume $G^{(i)}$ is connected. Given  $F_1,F_2\in V(G_{\{i\}}(\Delta))$,   there exist $k_1,k_2\in V(G^{(i)})$ with  $k_i\in F_i$ for $i=1,2$. Suppose that $k_1=k_2$. Then  $F_1, \{k_1\}, F_2$ is a path in $G_{\{i\}}(\Delta)$, and thus $F_1$ is connected to $F_2$.  Next suppose that $k_1\neq k_2$.  Since  $G^{(i)}$ is connected, there is a path $k_1=j_0, j_1, \ldots, j_s=k_2$ in $G^{(i)}$. Note that $\{j_{\ell},j_{\ell+1}\}\in V(G_{\{i\}}(\Delta))$ for all $\ell=0,\ldots,s-1$. Therefore, \[
F_1,\{k_1\},\{j_0,j_1\},\{j_1\},\{j_1,j_2\},\{j_2\},\ldots,\{j_{s-1},j_s\},\{k_2\},F_2
 \]
 is a path  in $G_{\{i\}}(\Delta)$, and so $F_1$ is connected to $F_2$. It follows that $G_{\{i\}}(\Delta)$ is connected.

Conversely, assume that $G_{\{i\}}(\Delta)$ is connected. Given  $k_1,k_2\in V(G^{(i)})$, there is a path $\{k_1\},F_1,F_2,\ldots,F_{2t-1},\{k_2\}$ in $G_{\{i\}}(\Delta)$. Hence
\[
\{k_1\}\subseteq F_1\supseteq F_2 \subseteq F_3\supseteq \ldots \subseteq F_{2t-1}\supseteq \{k_2\}.
\]
We use the induction on $t$ to show that there is a path from $k_1$ to $k_2$ in $G^{(i)}$. If  $t=1$ then $\{k_1,k_2\}\subseteq F_1$ and so $k_1$ is adjacent to $k_2$ (in $G^{(i)}$). For $t>1$, let $k_0\in F_2\cap V(G^{(i)})$.  Then $k_1$ is adjacent to $k_0$,  and by induction hypothesis there is a path in $G^{(i)}$  from $k_0$ to $k_2$. Hence  there is a path in $G^{(i)}$   from $k_1$ to $k_2$ and it follows that  $G^{(i)}$ is connected. \end{proof}

Considering  the proof of Theorem~\ref{combinatorially} one even shows that the graphs $G_{\{i\}}(\Delta)$ and $G^{(i)}$ have the same number of connected components.

\begin{Corollary}
\label{triangle}
 If $G$ contains no  triangle, then $G$ is inseparable.
\end{Corollary}

\bigskip
\noindent
{\it The conditions $(\alpha)$ and $(\beta)$.}
Let $G$ be a finite graph on the vertex set $[n]$ and $\Delta(G)$ be the independence complex of $G$.
Let $A$ be a subset of $[n]$.  Then $\link A$ of $\Delta(G)$ can be interpreted as the independence complex of a suitable graph if $A$ is an independent subset of $G$. In order to show this we introduce some notation.

The set
$$N(A)=\bigcup_{i\in A}N(i)$$
is called the {\em neighborhood} of $A$ (in $G$), and the set
$$N[A]=A\cup N(A)$$
is called the {\em closed neighborhood} of $A$ (in $G$).

Let $B\subset [n]$. The {\em induced subgraph} of $G$ with vertex set $B$, is the graph $G_{B}$ with edges $\{i,j\}\in E(G)$ and such that $i,j\in B$. An {\em induced cycle} of $G$ is a cycle of $G$ which is of the form $G_{B}$.
By $G\setminus A$ we denote the induced subgraph of $G$ on the vertex set $[n]\setminus A$.

\begin{Lemma}
\label{lem-graphLink}
Let $A\subseteq [n]$ be an  independent subset of $G$.
Then $\link_{\Delta(G)} A$ is the independence complex
of the graph $G\setminus N[A]$.
\end{Lemma}

\begin{proof}
Note that $F\in \link_{\Delta(G)} A$ if and only if
$F\subseteq [n]\setminus A$ and $F\cup A\in \Delta(G)$. The last condition is equivalent to saying that
$F\cup A$ does not contain any edge of $G$. Thus $F\in \link_{\Delta(G)} A$ if and only if
$F\subseteq [n]\setminus N[A]$
and $F$ does not  contain any edge of $G$. Since the set of edges of $G$  in $[n]\setminus N(A)$ is the same as the set of edges of $G$ in $G\setminus N(A)$, the desired conclusion follows.
\end{proof}

For a given subset  $B\subseteq [n]$,
we may easily express  the sets
$$\widetilde{N}_B(\Delta(G))\subseteq N_B(\Delta(G)) \subseteq [n]$$ in terms of $G$:
\[
N_B(\Delta(G))=\{F\subseteq [n]\:\; F\cap B=\emptyset,\;
F\text{ contains no edges of } G, \text{ but $F\cup B$ does}\},
\]
and
$
\widetilde{N}_B(\Delta(G))=
$
\[
\{F\in N_B(\Delta(G))\:\;  \text{there exists  $B'\subsetneq B$
 such that  $F\cup B'$  contains an edge of  $G$}\}.
\]

The following lemma lists some obvious properties of these sets.

\begin{Lemma}
\label{lem-graphNb} Let $B\subseteq [n]$. Then the following statements hold:

{\em (a)} If $|B|\geq 3$ or $|B|=2$ and $B$ is not an edge of $G$,
then $N_B(\Delta(G))=\widetilde{N}_B(\Delta(G))$.

{\em (b)} If $B$ is an edge, then $\emptyset\in N_B(\Delta(G))$,
and $\widetilde{N}_B(\Delta(G))=\emptyset$ if and only  $B$ is an isolated edge of  $G$,
i.e.\ it does not have a common vertex with  any other edge of $G$.
\end{Lemma}

Combining this lemma with Corollary~\ref{useful} we obtain

\begin{Corollary}
\label{cor-bBig}
Let $\bb\in \{0,1\}^{n}$ and let $B=\supp \bb$. Suppose that $|B|\geq 2$. Then $T^1(\Delta(G))_{-\bb}=0$
unless $B$ is an isolated edge in $G$. On the other hand, if $B$ is an isolated edge in $G$, then
$T^1(\Delta(G))_{-\bb}$ is one-dimensional.
\end{Corollary}

Let $G$ be a graph on the vertex set $[n]$. Based on  Theorem~\ref{combinatorially} and Corollary~\ref{cor-bBig} we see that $\Delta(G)$ is $\emptyset$-rigid, (i.e., $T^1(\Delta(G))_{-\bb}=0$ for every $\bb\in \{0,1\}^n$) if and only if $G^{(i)}$ is connected for all $i\in [n]$ and $G$ contains no isolated edge. Combining this fact with Proposition~\ref{Proposition11} and Lemma~\ref{lem-graphLink} we obtain  the following combinatorial conditions  for a graph to be rigid:
the graph  $G$ is rigid if and only if for all independent sets $A\subseteq V(G)$ one has:
\begin{enumerate}
\item[($\alpha$)] $(G\setminus N[A])^{(i)}$ is connected for all $i\in [n]\setminus N[A]$;

\item[($\beta$)] $G\setminus N[A]$ contains no isolated edge.
\end{enumerate}

It is obvious from Corollary~\ref{triangle}  that any bipartite graph is inseparable and  so it satisfies the condition $(\alpha)$, since any induced graph of a bipartite graph is again bipartite.  But, by far, not all bipartite graphs  are rigid. For example  we have

\begin{Corollary}
\label{CM}
Let $G$ be a Cohen--Macaulay bipartite graph (i.e., $S/I(G)$ is Cohen--Macaulay). Then $G$ is not rigid.
\end{Corollary}

\begin{proof} By Proposition~\ref{join} we may assume that $G$ is connected. By \cite[Theorem 3.4]{HH} the graph $G$, after a suitable relabeling of its vertices, arises from a finite poset $P=\{p_1,\ldots,p_n\}$  as follows: $V(G)=\{p_1,\ldots,p_n, q_1,\ldots,q_n\}$ and $E(G)=\{\{p_i,q_j\}\:\; p_i\leq p_j\}$. We may assume that $p_1$ is a minimal element in $P$. Let $A=\{p_2,\ldots,p_n\}$. Then  $N[A]=\{p_2,\ldots,p_n,q_2,\ldots,q_n\}$, and $G\setminus N[A]=\{p_1,q_1\}$. It follows from ($\beta$)  that $G$ is not rigid.
\end{proof}

Recall that a vertex $v$ is called a {\em free vertex} if $\deg v=1$, and an edge $e$ is called a {\em leaf} if it has a free vertex. An edge $e$ of $G$ is called {\em branch}, if there exists a leaf $e'$  with $e'\neq e$ such that $e\sect e'\neq \emptyset$.

Let $e=\{i,j\}\in E(G)$. We denote by  $ N_0(e)$ the set $(N(i)\cup N(j))\setminus\{i,j\}$.

\medskip
In the next proposition we  present  sufficient conditions for graph to satisfy  $(\alpha)$ or ($\beta$).

\begin{Proposition}
\label{sufficient}
Let $G$ be a finite graph on the vertex set $[n]$.
\begin{enumerate}
\item[ (a)] Suppose that each vertex of a $3$-cycle in $G$ belongs to a leaf. Then  $G$ satisfies $(\alpha)$, and hence $G$ is inseparable. In particular, if $G$ does not contain any $3$-cycle, then $G$ is inseparable.
\item[(b)] Suppose that each edge of $G$ is a branch. Then $G$ satisfies $(\beta)$.
\item[ (c)]  Suppose that each edge of $G$ is a branch and each vertex of a $3$-cycle of $G$ belongs to a leaf. Then $G$ is rigid.
\end{enumerate}
\end{Proposition}
\begin{proof}
(a) Suppose that $(\alpha)$ is not satisfied. Then   there exists an independent set $A\subseteq [n]$ and $i\in [n]\setminus N[A]$ such that $N(i)\sect [n]\setminus N[A]=B\cup C$ with $B, C\neq \emptyset$ and $B\sect C=\emptyset$ and $\{j,k\}\in E(G\setminus N[A])\subseteq E(G)$ for all $j\in B$ and all $k\in C$. Since $B, C\neq \emptyset$ there exist $j\in B$ and  $k\in C$ such that $\{j,k\}\in E(G)$. Thus $i$ is a vertex of a $3$-cycle in $G$. By assumption  there exists a leaf  $\{i,t\}$ in $G$. Suppose that $t\in  N[A]$. If $t\in A$ then $i\in  N[A]$, a contradiction. Thus $t\in  N(A)$. Since $\deg t=1$ it follows that $i\in A$, again a contradiction. Therefore we see that $t\in [n]\setminus  N[A]$. This implies that $t\in B\cup C$. We may assume that $t\in B$. So $\{t,k\}\in E(G)$, a contradiction since $N(t)=\{i\}$.

It follows from Theorem~\ref{combinatorially} that $G$ is inseparable if it satisfies $(\alpha)$.
Suppose now that $G$ does not contain any $3$-cycle. Then, by the first part of the statement, $G$ satisfies condition $(\alpha)$, and so it is inseparable.

(b) Suppose that $G$ does not satisfy ($\beta$). Then there exists an independent set $A$ of the vertices of $G$ such that $G\setminus  N[A]$ contains an isolated edge, say $e=\{i,j\}$.   We show that $e$ is not a branch  of $G$ and so we get a contradiction. Let $v\in  N_0(e)$ and let $\deg v=1$. Then we have  either $N(v)=\{i\}$ or $N(v)=\{j\}$. Without loss of generality we may assume that $N(v)=\{i\}$. Since $e$ is an isolated edge in $G\setminus  N[A]$ we have $v\in   N[A]$. Suppose that $v\in A$. Then  $\{i\}=N(v)\subseteq  N(A)$, a contradiction. Thus  $v\in  N(A)$, and so there exists $t\in A$ such that $v\in N(t)$. Since $N(v)=\{i\}$ it follows that $t=i$. This implies that $i\in A$,  which is again a contradiction.   Hence $\deg v\geq 2$, as desired.

(c) follows from (a) and (b).
\end{proof}

\bigskip
\noindent
{\it Rigid graphs.}
The next two lemmata will help us to classify the rigid chordal  graphs and rigid graphs without induced cycles of length $4$,$5$ or $6$.

Recall that a graph $G$ is {\em chordal} if any cycle of length $\geq 4$ has chord. A chord of a cycle $C$ is an edge $\{i,j\}$ of $G$ with $i,j\in V(C)$ which is not an edge of $C$

\begin{Lemma}
\label{induced}
Let $G$ be a  rigid graph on the vertex set $[n]$, and let $A$ be an independent set of  vertices of $G$. Then $G\setminus  N[A]$ is rigid.
\end{Lemma}

\begin{proof}
Let $B$ be an independent set of  vertices of $G\setminus  N[A]$. Then $A\cup B$ is an independent set of vertices of $G$. Indeed, suppose that $\{i,j\}\in E(G)$ for some $i\in A$ and some $j\in B$. Then $j\in N(i)$ implies that $j\in  N(A)\subseteq  N[A]$. Since $ N[A]\sect B=\emptyset$ it  follows that $j\notin B$ which is a contradiction.

Clearly, $ N[A]\cup N[B]=N[A\cup B]$. Thus, $([n]\setminus  N[A])\setminus N[B]=[n]\setminus N[A\cup B]$. Since for any subset $C$ of $[n]$,  $G\setminus C$ is an induced subgraph of $G$, we have $(G\setminus  N[A])\setminus N[B]=G\setminus N[A\cup B]$. Since rigidity is characterized by $(\alpha)$ and $(\beta)$,  the statement follows.
\end{proof}

Let $G$ be a   graph on the vertex set $[n]$. For each $e=\{i,j\}\in E(G)$ we define the set $O_G(e)$ as follows:
\begin{eqnarray*}
O_G(e)=\{v':\ v'\in \bigcup_{v\in  N_0(e)}N(v), \ N(v')\sect \{i,j\}=\emptyset\}.
\end{eqnarray*}

\begin{Lemma}
\label{inducedrigid}
Let $G$ be a  rigid graph on the vertex set $[n]$ which does not contain any induced $4$-cycle, and let $e=\{i,j\}$ be an edge of $G$ which is not a branch. Then $O_G(e)\neq \emptyset$.
\end{Lemma}

\begin{proof}
First we show that  for all $v\in  N_0(e)$ we have $N(v)\setminus\{i,j\}\neq \emptyset$. Note that since $G$ is rigid the edge $e$ is not isolated, and so $ N_0(e)\neq \emptyset$.  Suppose  that there exists $v\in  N_0(e)$ such that $N(v)\subseteq \{i,j\}$. Without loss of generality we may assume that $v\in N(i)$. Since $e$ is not a branch we have $\deg v\geq 2$. It follows that $N(v)=\{i,j\}$.
Therefore $G^{(v)}$ consists of two isolated vertices $i$ and $j$ which contradicts the fact that $G$ is rigid.

Now suppose that $O_G(e)=\emptyset$, i.e., for all $v\in  N_0(e)$ and for all $v'\in N(v)$ we have $N(v')\sect \{i,j\}\neq\emptyset$. Since $ N_0(e)\neq \emptyset$ we may assume that there exists $v\in N(i)$ with $v\neq j$. As shown above $N(v)\setminus \{i,j\}\neq \emptyset$.
Suppose that for all  $v'\in N(v)\setminus \{i,j\}$
 we have $i\in N(v')$. Then $i$ is an isolated vertex in $G^{(v)}$, a contradiction. Thus there exists
$v'\in N(v)\setminus \{i,j\}$ such that $v'\notin N(i)$. Hence $v'\in N(j)$ which  implies that $v\in N(j)$ because $G$ does not contain any induced $4$-cycle. Since $v'\in N(j)$ and $v'\neq i,j$ we have $v'\in  N_0(e)$. So $N(v'')\sect \{i,j\}\neq \emptyset$ for all $v''\in N(v')$. As shown above $N(v')\setminus \{i,j\}\neq \emptyset$. Suppose that there exists  $v''\in N(v')\setminus \{i,j\}$
 such that  $v''\notin N(j)$. Then $v''\in N(i)$. It follows that $G$ contains the induced $4$-cycle with vertices $i$, $j$, $v'$ and $v''$,  a contradiction. Consequently,
   $v''\in N(j)$  for all $v''\in N(v')\setminus \{i,j\}$. Then $j$ is an isolated vertex in $G^{(v')}$, a contradiction.
This completes the proof.
\end{proof}

\begin{Theorem}
\label{chordal}
Let $G$ be a   graph on the vertex set $[n]$ such that  $G$ does not contain any induced cycle of length $4$, $5$ or $6$. Then $G$ is   rigid  if and only if each edge of $G$ is a branch and each vertex of a $3$-cycle of $G$  belongs to a leaf.
\end{Theorem}

\begin{proof}
 By using part (c) of Proposition~\ref{sufficient} it is enough to show that if a rigid graph  $G$  does not contain any induced cycle of length $4$, $5$ or $6$, then  each edge of $G$ is a branch and each vertex of a $3$-cycle of $G$  belongs to a leaf.

Suppose that $e=\{i,j\}\in E(G)$ is not a branch.  By Lemma~\ref{inducedrigid}, $O_G(e)\neq \emptyset$.  We claim that there exists $A\subseteq O_G(e)$ such that $A$ is independent in $G$ and $e$ is an isolated edge in $G\setminus  N[A]$. This will imply that $G$ is not rigid, a contradiction.
Let $G'$ be the induced subgraph  of $G$ on the vertex set  $O_G(e)$ and let $C_1, \ldots, C_m$ be   the connected components of $G'$. Let $u,v\in V(C_k)$ such that  $\{u,v\}\in E(G)$. We show that either  $N(u)\cap  N_0(e)\subseteq N(v)\cap N_0(e)$ or $N(v)\cap  N_0(e)\subseteq N(u)\cap  N_0(e)$.

Assume that $N(u)\cap  N_0(e)\not\subseteq N(v)\cap  N_0(e)$. Then there exists  $x\in N(u)\cap  N_0(e)$ such that $\{v,x\}$ is not an edge in $G$. Without loss of generality  we may assume that $i\in N(x)$.

Let $y\in N(v)\cap N_0(e)$, and first suppose that  $y\in N(i)$. Then we have the  $5$-cycle with vertices $i$, $x$, $u$, $ v$ and $y$.  Since $v,u\in O_G(e)$ it follows that $\{u,i\}, \{v,i\} \notin E(G)$ and since $\{v,x\}\notin E(G) $ it follows that $\{u,y\}\in E(G)$ because $G$ does not contain any induced cycle of length $4$ and $5$. Therefore $y\in N(u)\sect  N_0(e)$. On the other hand if  $y\in N(j)$, then we  have the $6$-cycle  with vertices $i$, $x$, $u$, $v$, $y$ and $j$. Note that  $\{v,x\},\{u,i\}, \{u,j\}, \{v,i\}, \{v,j\}\notin E(G)$. This implies that  $\{v,y\}\in E(G)$ since $G$ does not contain any induced cycle of length $4$, $5$ and $6$. Thus either  $N(u)\cap N_0(e)\subseteq N(v)\cap  N_0(e)$ or $N(v)\cap N_0(e)\subseteq N(u)\cap  N_0(e)$, as desired.

Now given $C_k$ we choose a maximal set $D_k=\{u_{1}, \ldots, u_{l}\}\subseteq V(C_k)$ with the property that the sets $N(u_{r})\sect  N_0(e)$ are pairwise different. After having defined the set $D_k$ for each $C_k$ we are ready to define the set $A$.

We let $A$ be the unique subset of $O_G(e)$ such that $A\sect C_k$ consists of all elements $u_r\in D_k$ with the property that $N(u_r)\sect  N_0(e)\not\subseteq N(u_s)\sect N_0(e)$ for all $u_s\in D_k$ with $s\neq r$.

In order to complete the proof we  show that $A$ is independent in $G$ and $e$ is an isolated edge in $G\setminus  N[A]$. Let $u,v\in a$ and assume that $\{u,v\}\in E(G)$. Then there exists $k$ such that $u,v\in D_k\subseteq V(C_k)$. Therefore either  $N(u)\cap  N_0(e)\subseteq N(v)\cap  N_0(e)$ or $N(v)\cap  N_0(e)\subseteq N(u)\cap  N_0(e)$. Thus by the choice of $A$, it follows that $u=v$, a contradiction. So $A$ is an independent set of $G$.

 Finally  we show that $e$ is an isolated edge of $G'':=G\setminus  N[A]$.  Note that for any $v\in  N_0(e)\sect V(G'')$ and for any $v'\in N(v)\setminus \{i,j\}$ we have $N(v')\sect \{i,j\}\neq \emptyset$. In fact, suppose that there exists $v_1\in N(v)$ such that $N(v_1)\sect \{i,j\}= \emptyset$. So $v_1\in V(C_k)$ for some $k$. If $v_1\in A$, then $v\in  N(A)\subseteq  N[A]$, a contradiction. Hence $v_1\notin A$.  Therefore, by the choice of $A$, there exists $v_2\in A$ such that $N(v_1)\sect N_0(e)\subseteq N(v_2)\sect  N_0(e)$. Since $v\in N(v_1)\sect  N_0(e)$ we have $v\in N(v_2)$ and so $v\in  N(A)\subseteq  N[A]$, a contradiction.  This shows that $O_{G''}(e)=\emptyset$.

 Suppose that $e$ is not  an isolated edge of $G''$, i.e., $ N_0(e)\sect V(G'')\neq\emptyset$. We observe that $e$ is not a branch in $G''$. Indeed,  if $e$ is a branch, then since  $e$ is not  isolated, there exists $v\in  N_0(e)\sect V(G'')$ such that degree of $v$ in $G''$ is one. We may assume that $v\in N(i)\setminus N(j)$. Since $e$ is not a branch in $G$ we have $N(v)\setminus \{i,j\}\neq \emptyset$ and for any $v'\in N(v)\setminus \{i,j\}$ we have $v'\in  N[A]$. This implies that  for any $v'\in N(v)$, $v'\in  N(A)$ because if $v'\in a$, then $v\in  N(A)\subseteq  N[A]$ and hence $v\notin V(G'')$, a contradiction. As seen in the previous paragraph, for any $v'\in N(v)\setminus \{i,j\}$ we have  $N(v')\sect \{i,j\}\neq \emptyset$. If $v'\notin N(i)$, then $v'\in N(j)$. Since $v\notin N(j)$ we will get the induced $4$-cycle with the vertices $i$, $v$, $v'$ and $j$, a contradiction. So $v'\in N(i)$ for any $v'\in N(v)$. It follows that $i$ is an isolated vertex in $G^{(v)}$, a contradiction. Thus $e$ is not a branch in $G''$.

Lemma~\ref{induced} implies that $G''$ is rigid and hence by Lemma~\ref{inducedrigid} it follows that $O_{G''}(e)\neq \emptyset$, a contradiction. So indeed $e$ is isolated in $G''$.

Now we prove that each vertex of a $3$-cycle of $G$  belongs to a leaf. Suppose that there exists $i\in [n]$, $i$ belongs to a $3$-cycle in $G$ and it does not belong to a leaf. So for all $v\in N(i)$ we have $\deg v\geq 2$.
Let $j$ and $k$ be  the two other vertices of this $3$-cycle. If $N(i)=\{j,k\}$, then $j$ and $k$ are isolated vertices of $G^{(i)}$,  contradicting  $(\alpha)$. So $N(i)\setminus \{j,k\}\neq \emptyset$.  Since each edge of the graph $G$ is a branch, for any $v\in N(i)$  the edge  $\{i, v\}$ is a branch. Since $i$ does not belong to a leaf it follows that any $v\in N(i)$ belongs to a leaf. Thus  for any $v\in N(i)$ there exists $i_v\in N(v)$ with $\deg i_v=1$.   Set $a=\bigcup_{v\in N(i)\setminus \{j,k\}}\{i_v\}$. Clearly, $A$ is an independent set of the vertices of $G$ and $j$, $k$ are two isolated vertices in $(G\setminus  N[A])^{(i)}$, a contradiction. Consequently, $i$ belongs to a leaf, as desired.
\end{proof}

 \begin{Corollary}
 \label{againchordal}
 Let $G$ be a chordal graph. Then $G$ is   rigid  if and only if each edge of $G$ is a branch and each vertex of a $3$-cycle of $G$  belongs to a leaf.
 \end{Corollary}

\begin{Corollary}
\label{forest}
Let $G$ be a graph on the vertex set $[n]$.
Suppose that  all cycles of $G$ have length $\geq7$ (which for example is the case when $G$ is a forest). Then $G$ is rigid  if and only if each edge of $G$ is a branch.
\end{Corollary}

\begin{proof}
Since $G$ does not contain any $3$-cycle, the statement follows from Theorem~\ref{chordal}.
\end{proof}

As another application we have
\begin{Corollary}
\label{rigidcycle}
Let $C$ be a cycle. Then $C$ is rigid if and only if  $C$ is a $4$-cycle or a $6$-cycle.
\end{Corollary}

\begin{proof}
Suppose that $|C|\neq 4,5,6$. Then by Theorem~\ref{chordal}, $C$ is not rigid.  Suppose now that $|C|=5$. Then $C\setminus  N[A]$ is an isolated edge, where  $A=\{i\}$ for some vertex $i$ of $C$, and hence the condition $(\beta)$ is not satisfied. So $C$ is not rigid also when $|C|=5$. In conclusion, $C$ is not rigid if $|C|\notin \{4,6\}$.

Next suppose that $|C|\in \{4,6\}$. Since $C$ does not contain a $3$-cycle, it follows that $C$ satisfies ($\alpha$) by Proposition~\ref{sufficient}(a).

 Note that for any  nonempty independent subset $A$ of $V(C)$, $C\setminus N[A]$ is either an empty graph (i.e., a graph containing no edge) or a path of length 3. Therefore,   the cycle $C$ also satisfies the condition ($\beta$). Hence $C$ is rigid.
\end{proof}

{\noindent \bf Acknowledgement:} We would like to express our deep gratitude to the referee for his/her careful reading and for his/her excellent advices, which improved the presentation of this paper considerably.

\end{document}